\documentclass[10pt,british]{article}
\usepackage{amsfonts}
\usepackage{latexsym}
\usepackage{amsmath}
\usepackage{amssymb}
\usepackage{amssymb}
\usepackage{amsthm}
\usepackage{mathrsfs}
\usepackage[pagebackref,colorlinks=true]{hyperref}

\newtheorem{thrm}{Theorem}[section]

\newtheorem{lemma}[thrm]{Lemma}

\newtheorem{prop}[thrm]{Proposition}

\newtheorem{cor}[thrm]{Corollary}

\newtheorem{dfn}[thrm]{Definition}

\newtheorem{rmrk}[thrm]{Remark}

\newtheorem{conv}[thrm]{Convention}

\newcommand{\newsection}{    
\setcounter{equation}{0}\section}
\def\appendix#1{\addtocounter{section}{1}\setcounter{equation}{0}
\renewcommand{\thesection}{\Alph{section}}
\section*{Appendix \thesection\protect\indent \parbox[t]{11.15cm}{#1}}
\addcontentsline{toc}{section}{Appendix \thesection\ \ \ #1}}

\newcommand{\be}{\begin{eqnarray}}
\newcommand{\ee}{\end{eqnarray}}
\newcommand{\bea}{\begin{eqnarray}}
\newcommand{\eea}{\end{eqnarray}}
\newcommand{\ba}{\begin{array}}
\newcommand{\ea}{\end{array}}

\newenvironment{frcseries}{\fontfamily{frc}\selectfont}{}

\def\d{\delta}

\def\sb {{\nabla}}
\def\LC{{\nabla^g}}
\def\ps{{\Psi^+}}
\def\sp{{\Psi^-}}

\begin{document}
\begin{center}
\vspace*{-1.0cm}


{\Large \bf Riemannian curvature identities  on  almost  Calabi-Yau  with torsion   6-manifold  and generalized Ricci solitons
} 

\vspace{0.5  cm}
 {\large S.~ Ivanov${}^1$ and  N. Stanchev$^2$}

\vspace{0.5cm}

${}^1$ University of Sofia, Faculty of Mathematics and
Informatics,\\ blvd. James Bourchier 5, 1164, Sofia, Bulgaria
\\and  Institute of Mathematics and Informatics,
Bulgarian Academy of Sciences\\
email: ivanovsp@fmi.uni-sofia.bg

\vspace{0.5cm}
${}^2$ University of Sofia, Faculty of Mathematics and
Informatics,\\ blvd. James Bourchier 5, 1164, Sofia, Bulgaria\\


\end{center}

\vskip 0.5 cm
\begin{abstract}
It is observed that on a compact almost complex Calabi-Yau with torsion 6-manifold  the Nijenhuis tensor is parallel with respect to the torsion connection. If the torsion is closed then the space is a compact generalized  gradient Ricci soliton. In this case, the torsion connection is Ricci-flat if and only if either the norm of the torsion  or the Riemannian scalar curvature is constant. On a compact almost complex Calabi-Yau with torsion 6-manifold  it is shown that the curvature of the torsion connection  is symmetric  
 on exchange of the   first and the second pairs and has vanishing Ricci tensor if and only if it satisfies the Riemannian first Bianchi identity.
\medskip


\medskip

\end{abstract}

\vskip 0.5cm
\noindent{\bf Acknowledgements:} \vskip 0.1cm
We would like to thank Jeffrey Streets and Ilka Agricola for the extremely useful remarks,  comments and suggestions.

The research of S.I.  is partially supported by Contract KP-06-H72-1/05.12.2023 with the National Science Fund of Bulgaria,  by Contract 80-10-192 / 17.5.2023   with the Sofia University "St.Kl.Ohridski" and the National Science Fund of Bulgaria, National Scientific Program ``VIHREN", Project KP-06-DV-7. The research of N.S.  is partially  financed by the European Union-Next Generation EU, through the National Recovery and Resilience Plan of the Republic of Bulgaria, project 
SUMMIT BG-RRP-2.004-0008-C01.

\vskip 0.5cm

Statements and Declarations: not applicable


\tableofcontents

\setcounter{section}{0}
\setcounter{subsection}{0}



\newsection{Introduction}
Riemannian manifolds with metric connections having totally skew-symmetric torsion and special holonomy received a lot of interest in mathematics and theoretical physics mainly from supersymmetric string theories and supergravity.  The main reason comes from the Hull-Strominger system which describes the supersymmetric background in heterotic string theories \cite{Str,Hull}. The number of preserved supersymmetries depends on the number of parallel spinors with respect to a metric connection $\sb$ with totally skew-symmetric torsion $T$. 
The torsion 3-form $T$ is identified with the 3-form field strength in these theories. 
The presence of a $\nabla-$parallel spinor leads to restriction of the
holonomy group $Hol(\nabla)$ of the torsion connection $\nabla$.
Namely, $Hol(\nabla)$ has to be contained in $SU(n), dim=2n$
\cite{Str,GMW,IP1,IP2,Car,BB,BBE,GIP}, the exceptional group $G_2,
dim=7$ \cite{FI,GKMW,FI1}, the Lie group $Spin(7), dim=8$
\cite{GKMW,I1}. A detailed analysis of the possible geometries is
carried out in \cite{GMW}.

Complex non-K\"ahler geometries appear
in string compactifications  and are studied intensively for a long time
\cite{Str,GP,GMW,GKMW,GMPW,GP,GPap,BB,BBE}. 
Hermitian manifolds have widespread applications in both physics and differential geometry in connection with solutions to the Hull-Strominger system see \cite{LY,yau,yau1,FIUV,XS,OLS,PPZ,PPZ1,PPZ2,PPZ3,PPZ4,FHP,FHP1,CPYau,CPY1,GRST,Ph} and references therein.
On a Hermitian manifold,  there exists   a  unique
connection which preserves  the Hermitian structure and has
totally skew-symmetric torsion tensor. Its existence and explicit expression first appeared in Strominger's seminal paper \cite{Str} in 1986 in connection with the heterotic supersymmetric string background, where he called it the H-connection. Three years later, Bismut formally discussed
and used this connection in his local index theorem paper \cite{bismut}, which leads to the name Bismut
connection in literature.
We call this connection the Strominger-Bismut connection. Note that the connection also appeared implicitly earlier (\cite{Yano}) and in some literature it was also called the KT connection (K\"ahler with torsion) or characteristic connection.  When the holonomy of the Strominger-Bismut connection is contained in $SU(n)$ one has the notion of Calabi-Yau manifold with torsion (CYT) spaces and these manifolds are of great interest in string theories and in mathematics since CYT appear in the Hull-Strominger system. 
Vanishing theorems for
the Dolbeault cohomology on compact Hermitian non-K\"ahler manifold were found in terms
of the Strominger-Bismut connection in  \cite{AI,IP2,IP1,YZZ,Ye}.

If the torsion 3-form $T$ of $\nabla$ is closed, $dT=0$ which, in Hermitian manifold,  is equivalent to the condition $\partial\bar\partial F=0$, where $F$ is the fundamental 2-form,  the Hermitian metric g is called SKT (strong K\"ahler with torsion) \cite{hkt} or pluriclosed.
The SKT (pluriclosed) metrics have found many applications in
both physics, see eg \cite{hull, howe, Str, hethor,
hethor1,sethi} and geometry, see eg \cite{yaujost,IP1, IP2, salamon,  yau,
yau1, FIUV,FT1, streets,Str1} and references therein. For example in type II string theories, the three form
$T$ is identified with the 3-form field strength and it is required
by construction to satisfy $dT=0$ (see e.g. \cite{GKMW,GMW}). In a series of papers Streets and Tian
\cite{streets,ST,ST1} introduced a Hermitian Ricci flow under which the
pluriclosed or equivalently strong KT structure is preserved, show the relationship of pluriclosed flow to type II string backgrounds and find that pluriclosed flow preserves N=(2,2) supersymmetry (generalized K\"ahler geometry). More generally, the geometry of a torsion connection with closed torsion form appears in the framework of the generalized Ricci flow and the generalized (gradient) Ricci solitons developed by Garcia-Fernandez and Streets \cite{GFS} (see the references therein). Compact pluriclosed (strong) CYT spaces are connected with the pluriclosed flow and the equations of motion in compactifications of type II supergravity due to the recent work of Garcia-Fernandez, Jordan and Streets \cite{GFJS}.  It is shown in \cite[Proposition~8.14]{GFS}, (see also \cite[Proposition~2.6]{GFJS} that any compact pluriclosed CYT space is automatically a steady generalized gradient Ricci soliton  \cite[Proposition~8.14]{GFS}, (see also \cite[Proposition~2.6]{GFJS}). 
Moreover, global existence and convergence of pluriclosed flow to a Stromonger-Bismut-flat metric on complex manifolds admitting Strominger-Bismut-flat metric is also shown in \cite{GFJS}.
Generalizations of the pluriclosed condition $\partial\bar\partial F=0$ on 2n dimensional Hermitian manifolds in the form $
F^\ell\wedge
\partial\bar\partial F^k=0~, ~~~1\leq k+\ell\leq n-1~$  has been investigated in \cite{FU,Pop,FWW,IP3} etc.

Hermitian metrics with the Strominger-Bismut connection being K\"ahler-like,
namely, its curvature satisfies  the Riemannian first Bianchi identity (\eqref{RB} below) have been studied in \cite{AOUV},
investigating this property on 6-dimensional solvmanifolds with holomorphically
trivial canonical bundle and it is conjectured by Angela-Otal-Ugarte-Villakampa in \cite{AOUV} that such metrics should be SKT (pluriclosed).
Support to that can be derived from  [Theorem~5]\cite{WYZ} where the simply connected Hermitian manifolds with flat  Strominger-Bismut connection are classified and it easily follows that it is SKT (pluriclosed) with parallel 3-form torsion \cite{AOUV}. The latter conclusion also follows from the Cartan-Schouten theorem \cite{CS}, (see also [Theorem~2.2]\cite{AF}).
The conjecture by  Angella-Otal-Ugarte-Villakampa was proved in \cite{ZZ}. It is also shown in \cite{ZZ} that in this case the torsion is closed and $\sb-$parallel, $dT=\nabla T=0$. 
Hermitian nilmanifolds whose Strominger-Bismut connection is  K\"ahler-like are classified in \cite{ZZ1}. A classification of compact non-K\"ahler Hermitian manifolds with  K\"ahler-like Strominger-Bismut connection
in complex dimension 3 and those with degenerate torsion in higher dimensions is given in \cite{YZZ}. For complex dimension three \cite[Theorem~5]{YZZ} shows that if the space is CYT and the Strominger-Bismut connection is  K\"ahler-like then the Strominger-Bismut connection is flat and it is conjectured that this is true in higher dimensions provided the space is compact, \cite[Conjecture~2]{YZZ}.  This conjecture was confirmed very recently in \cite{ZZ3}.

 Some types of non-complex 6-manifolds have also been invented  in the
string theory due to the mirror symmetry and T-duality
\cite{KST,GLMW,GM,Car,Car1,KLMS,KML,HLM}. 
Almost Hermitian manifolds with totally skew-symmetric Nijenhuis
tensor arise as target spaces of a class of (2,0)-supersymmetric
two-dimensional sigma models \cite{Pap}. For the consistency of
the theory, the Nijenhuis tensor has to be parallel with respect
to the torsion connection with holonomy contained in $SU(n)$. The known models are group manifolds as well as the nearly K\"ahler spaces. 
In general, an almost Hermitian manifold does not admit a metric connection preserving the almost Hermitian structure and having  totally skew-symmetric torsion. It is shown in \cite[Theorem~10.1]{FI} that this is equivalent to the condition that the Nijenhuis tensor of type (0,3) 
 is totally skew-symmetric, i.e. this is  the class $G_1$ in Gray-Hervella classification \cite{GrH}. In this case,  the connection is unique. This class contains the Hermitian manifolds with Strominger-Bismut connection as well as the nearly K\"ahler spaces. In the nearly K\"ahler case the torsion connection $\sb$ is the characteristic connection considered by Gray, \cite{gray} and the torsion $T$ and the Nijenhuis tensor are $\sb-$parallel, $\sb T=\sb N = 0$ (see e.g. \cite{Kir,BM}). If the holonomy of the torsion connection is contained in $SU(n)$, that is,  the Ricci 2-form $\rho$ of $\sb$  representing the first Chern class is identically zero,  then we have the notion of \emph{Almost Calabi-Yau manifold with torsion (briefly ACYT)}.

 As it was pointed out in \cite{Str}, the existence of a parallel spinor with respect to a metric connection $\sb$ with torsion 3-form, in dimension six,  leads to the restriction that its holonomy group has to be contained in $SU(3)$. This means one has to consider an $SU(3)$-structure, i.e. an almost Hermitian manifold $(M^6,g,J)$ with topologically trivial canonical bundle trivialized by a (3,0) with respect to $J$ form  $\Psi=\Psi^++\sqrt{-1}\Psi^-$\ endowed with a metric connection preserving the $SU(3)$ structure with totally skew-symmetric torsion. 
 Our main purposes are to investigate the curvature properties of the torsion connection on an ACYT 6-manifold.

We observe in Theorem~\ref{thnnew}  that on any compact 6-dimensional ACYT the Nijenhuis tensor is $\sb-$ parallel and deduce in Corollary~\ref{corclosT} that a compact ACYT 6-manifold has closed torsion exactly when $Ric =-\sb\theta$. This helps to prove our first main result,
\begin{thrm}\label{closTt}
Let $(M,g,J,\Psi)$ be a 6-dimensional  compact  ACYT space with closed torsion, $dT=0$.

The following conditions are equivalent:
\begin{itemize}
\item[a).] The torsion connection  is Ricci flat, $Ric=0$;
\item[b).] The norm of the torsion is constant, $d||T||^2=0$;
\item[c).] The Lee form is $\nabla-$parallel, $\sb\theta=0$;
\item[d).] The torsion connection has vanishing scalar curvature, $Scal=0$;
\item[e).] The Riemannian scalar curvature is constant, $Scal^g=const.$;
\item[f).] The Lee form is co-closed, $\delta\theta=0$.
\end{itemize}
In each of the six cases above the torsion is co-closed and therefore it is  a harmonic 3-form.
\end{thrm}
In the complex case, $N=0$,  we have
\begin{cor}
A compact pluriclosed CYT 6-manifold  is Strominger-Bismut Ricci flat if and only if either the torsion has a constant norm or the Riemannian scalar curvature is constant.

In particular, the torsion is harmonic.
\end{cor}
We show in Theorem~\ref{inf} that any compact ACYT 6-manifold with closed torsion 3-form is a steady generalized gradient Ricci soliton and this condition is equivalent to a certain vector field to be parallel with respect to the torsion connection. We also find out that this vector field preserves the $SU(3)$ structure, thus extending the above-mentioned result in \cite{GFJS} for compact pluriclosed CYT space to compact ACYT space with closed torsion in dimension six.

Studying the curvature properties of the torsion connection on an ACYT 6-manifold we find sufficient conditions for the torsion 3-form to be closed and harmonic. 
\begin{thrm}\label{mainsu3}
Let $(M,g,J,\Psi)$ be a 6-dimensional ACYT space. 
The torsion connection $\sb$ preserving the $SU(3)$ structure 
is Ricci flat,  has curvature $R \in S^2\Lambda^2$ and the Nijenhuis tensor is of  constant norm,
\begin{equation}\label{s2l2}
R(X,Y,Z,V)=R(Z,V,X,Y),\qquad Ric(X,Y)=0, \qquad ||N||^2=const.
\end{equation}
if and only if the torsion  $T$ is $\LC$ and $\sb-$parallel,  
$\LC T=\sb T=0.$

In particular, the torsion 3-form is harmonic, 
the Nijenhuis tensor and the exterior derivative $dF$ of the K\"ahler 2-form are $\sb-$parallel, 
\begin{equation}\label{nsc}
\sb N=\sb dF=0.
\end{equation}
The curvature of the torsion connection obeys the Riemannian first Bianchi identity \eqref{RB} and vice-versa.
\end{thrm}
Taking into account  Theorem~\ref{thnnew} below, in the compact case,  Theorem~\ref{mainsu3} can be stated in the form
\begin{thrm}\label{cmainsu3}
Let $(M,g,J,\Psi)$ be a compact 6-dimensional ACYT space.
The torsion connection $\sb$ preserving the $SU(3)$ structure 
has curvature $R \in S^2\Lambda^2$ with vanishing Ricci tensor 
\begin{equation}\label{cs2l2}
R(X,Y,Z,V)=R(Z,V,X,Y),\qquad Ric(X,Y)=0. 
\end{equation}
if and only if the torsion 3-form $T$ is $\LC$ and  $\sb-$parallel, $\LC T=\sb T=0.$

In particular the torsion 3-form is harmonic,  $dT=\delta T=0$ and  \eqref{nsc}  holds.

\noindent The curvature of the torsion connection obeys the Riemannian first Bianchi identity \eqref{RB} and vice-versa.
\end{thrm}
We can reformulate the above result as follows
\begin{cor}\label{cco1}
On a compact  ACYT manifold in dimension 6, the next conditions are equivalent:
\begin{itemize}
\item The torsion connection  has curvature $R \in S^2\Lambda^2$ with vanishing Ricci tensor;
\item The curvature of the torsion connection is  K\"ahler like;
\item The torsion is closed and parallel,  $dT=\LC T=\sb T=0$.
\end{itemize}
\end{cor}
Theorem~\ref{mainsu3} 
 helps to generalize  \cite[Theorem~5]{YZZ} as follows
\begin{thrm}\label{cytt}
Let $M^6$ be a CYT whose Strominger-Bismut connection $\sb$ satisfies \eqref{cs2l2}. Then it is Strominger-Bismut   K\"ahler-like and therefore is Strominger-Bismut flat.
\end{thrm}
We recall, \cite {GFS}, that the fixed points of the generalized Ricci flow are Ricci flat metric connections with harmonic torsion 3-form, $Ric=dT=\delta T=0$. 
In this direction, our results show that a 6-dimensional ACYT manifold with a constant norm of the Nijenhuis tensor  (in particular any compact ACYT manifold)  with Ricci flat torsion connection having curvature $R\in S^2\Lambda^2$ is a fixed point of the corresponding generalized Ricci flow. In particular, if the curvature of the torsion connection satisfies the Riemannian first Bianchi identity then it is a fixed point of the generalized Ricci flow.  In the compact case, any ACYT 6-manifold with closed torsion and either constant Riemannian scalar curvature or constant norm of the torsion is a fixed point of the generalized Ricci flow.

Note that Hermitian and almost Hermitian manifolds of type $G_1$ with $\sb-$parallel torsion 3-form  are investigated in  \cite{AFS,Scho}, recently in \cite{CMS,ZZ2,NZ1}  and a large number of examples are given there. 
On the Hermitian manifold, $N=0$, the condition $\sb T=0$ is precisely  expressed in terms of the curvature of the Strominger-Bismut connection $\sb$ and its traces ( \cite[Theorem~1.1]{ZZ2}). More generally, Riemannian manifolds equipped with metric connection $\sb$ with $\sb-$parallel  torsion 3-form are investigated in \cite{AFF,AFer,Ch1,Ch2}.
It is conjectured very recently that if a complete Hermitian manifold has Strominger-Bismut parallel curvature and torsion and zero all Ricci tensors then it has to be Strominger-Bismut flat, \cite[Conjecture~1.8]{NZ1}.  A special case with the additional conditions that the space is K\"ahler-like and CYT is proved in \cite[Proposition~6.8]{NZ1}. It is worth mentioning  \cite[Theorem~4.1]{AFF} which states that an irreducible complete and simply connected Riemannian manifold of dimension bigger or equal to $5$ with $\sb-$parallel and closed torsion 3-form, $\sb T=dT=0$ (which is equivalent to $\sb T=\sigma^T=0$ due to \eqref{dh}  below), is a simple compact Lie group or its dual non-compact symmetric space with biinvariant metric, and, in particular, the torsion connection is the flat Cartan connection.

\begin{rmrk}
 A combination of the above-mentioned  \cite[Theorem~1]{ZZ} with \cite[Theorem~4.1]{AFF} shows that an irreducible complete and simply connected Hermitian manifold of dimension bigger or equal to 6 which is K\"ahler-like should be a simple compact Lie group with flat  Strominger-Bismut connection. Moreover, due to Theorem~\ref{cmainsu3} and \cite[Theorem~4.1]{AFF},   irreducible simply connected compact 6-dimensional  ACYT space satisfying conditions \eqref{cs2l2}  does not exist since there are no simple compact Lie groups of dimension six.
\end{rmrk}
We also note that the main result  \cite[Theorem~1]{ZZ}  combined with \cite[Corollary~4.7]{IP2} yields in a straightforward way the next
\begin{cor}
Let $M$ be a 2n-dimensional compact Hermitian manifold and its Strominger-Bismut connection is  K\"ahler-like. If $M$ is conformally balanced (the Lee form is an exact form)  then the space is a K\"ahler manifold.
\end{cor}

\begin{conv}
Everywhere in the paper, we will make no difference between tensors and the corresponding forms via the metric as well as we will use Einstein summation conventions, i.e. repeated Latin indices are summed over.
\end{conv}

\section{Preliminaries}
In this section, we recall some known curvature properties of a metric connection with totally skew-symmetric torsion on Riemannian manifold as well as
the notions and existence of a metric linear connection preserving a given Almost Hermitian structure and having skew-symmetric torsion from \cite{I,FI,IS}.
\subsection{Metric connection with totally skew-symmetric  torsion and its curvature}
Let $(M,g,\sb)$ be a Riemannian manifold with a metric connection $\sb$ and $T(X,Y)=\sb_XY-\sb_YX-[X,Y]$ be its torsion. The torsion of type (0,3) is denoted by the same letter and it is defined by $T(X,Y,Z)=g(T(X,Y),Z)$. The torsion is totally skew-symmetric if $T(X,Y,Z)=-T(X,Z,Y)$.

On a Riemannian manifold $(M,g)$ of dimension $n$ any metric connection $\sb$ with totally skew-symmetric torsion $T$ is connected with the Levi-Civita connection $\sb^g$ of the metric $g$ by
\begin{equation}\label{tsym}
\sb^g=\sb- \frac12T.
\end{equation}
The exterior derivative $dT$ has the following  expression (see e.g. \cite{I,IP2,FI})
\begin{equation}\label{dh}
\begin{split}
dT(X,Y,Z,V)=(\nabla_XT)(Y,Z,V)+(\nabla_YT)(Z,X,V)+(\nabla_ZT)(X,Y,V)\\+2\sigma^T(X,Y,Z,V)-(\nabla_VT)(X,Y,Z),
 \end{split}
 \end{equation}
where the 4-form $\sigma^T$, introduced in \cite{FI}, is defined by
$ \sigma ^T(X,Y,Z,V)=\frac12\sum_{j=1}^n(e_j\lrcorner T)\wedge(e_j\lrcorner T)(X,Y,Z,V)$, 
   $(e_a\lrcorner T)(X,Y)=T(e_a,X,Y)$ is the interior multiplication and $\{e_1,\dots,e_n\}$ is a  basis. The properties of the 4-form $\sigma^T$ are studied in detail in \cite{AFF} where it is shown that $\sigma^T$ measures the `degeneracy' of the 3-form $T$.

The identity \eqref{tsym}  yields
\begin{equation}\label{tgt}
\LC T=\sb T+\frac12\sigma^T.
\end{equation}
 For the curvature of $\sb$ we use the convention $ R(X,Y)Z=[\nabla_X,\nabla_Y]Z -
 \nabla_{[X,Y]}Z$ and $ R(X,Y,Z,V)=g(R(X,Y)Z,V)$. It has the well-known properties
 \begin{equation}\label{r1}
 R(X,Y,Z,V)=-R(Y,X,Z,V)=-R(X,Y,V,Z).
 \end{equation}
   The first Bianchi identity for $\nabla$ can be written in the  form (see e.g. \cite{I,IP2,FI})
 \begin{equation}\label{1bi}
 \begin{split}
 R(X,Y,Z,V)+ R(Y,Z,X,V)+ R(Z,X,Y,V)\\
 =dT(X,Y,Z,V)-\sigma^T(X,Y,Z,V)+(\nabla_VT)(X,Y,Z).
 \end{split}
 \end{equation}
It is proved in \cite[p.307]{FI} that the curvature of  a metric connection $\sb$ with totally skew-symmetric torsion $T$  satisfies also the  identity
 \begin{equation*}
 \begin{split}
 R(X,Y,Z,V)+ R(Y,Z,X,V)+ R(Z,X,Y,V)-R(V,X,Y,Z)-R(V,Y,Z,X)-R(V,Z,X,Y)\\
 =\frac32dT(X,Y,Z,V)-\sigma^T(X,Y,Z,V),
 \end{split}
 \end{equation*}
 which combined with \eqref{1bi} yields
 \begin{equation}\label{1bi1}
 \begin{split}
R(V,X,Y,Z)+R(V,Y,Z,X)+R(V,Z,X,Y)= -\frac12dT(X,Y,Z,V)+(\nabla_VT)(X,Y,Z)
 \end{split}
 \end{equation}
\begin{dfn} We say that the curvature $R$ satisfies the Riemannian first Bianchi identity if
\begin{equation}\label{RB}
R(X,Y,Z,V)+R(Y,Z,X,V)+R(Z,X,Y,V)=0.
\end{equation}
\end{dfn}
 It is well known algebraic fact that \eqref{r1} and \eqref{RB} imply $R\in S^2\Lambda^2$, i.e.
 \begin{equation}\label{r4}
 R(X,Y,Z,V)=R(Z,V,X,Y)
 \end{equation}
 holds.
 \begin{rmrk}
 Note that, in general, \eqref{r1} and \eqref{r4} do not imply \eqref{RB}.
 \end{rmrk}
It is known that a metric connection $\sb$ with totally skew-symmetric torsion $T$ has curvature $R\in S^2\Lambda^2$, i.e. it satisfies \eqref{r4}  if and only if the covariant derivative of the torsion with respect to the torsion connection  is a 4-form.
\begin{lemma}\cite[Lemma~3.4]{I} The next equivalences hold for a metric connection with torsion 3-form
\begin{equation}\label{4form}
(\sb_XT)(Y,Z,V)=-(\sb_YT)(X,Z,V) \Longleftrightarrow R(X,Y,Z,V)=R(Z,V,X,Y) ) \Longleftrightarrow dT=4\LC T.
\end{equation}
\end{lemma}
We need the very recent result from \cite{IS}
\begin{thrm} \cite[Theorem~1.2]{IS}\label{tFBI}
A metric connection $\sb$ with torsion 3-form $T$ satisfies the Riemannian first Bianchi identity exactly when the next identities hold
\begin{equation*}
 dT=-2\nabla T=\frac23\sigma^T.
\end{equation*}
In this case, the torsion 3-form $T$ is parallel with respect to a metric connection with torsion 3-form $\frac13T$ \cite{AF} and therefore has a constant norm, $||T||^2=const.$
\end{thrm}

 The   Ricci tensors and scalar curvatures of the connections $\LC$ and $\sb$ are related by \cite[Section~2]{FI}, (see also \cite [Prop. 3.18]{GFS})
\begin{equation}\label{rics}
\begin{split}
Ric^g(X,Y)=Ric(X,Y)+\frac12 (\delta T)(X,Y)+\frac14\sum_{i=1}^n(g(T(X,e_i),T(Y,e_i);\\
Scal^g=Scal+\frac14||T||^2,\qquad Ric(X,Y)-Ric(Y,X)=-(\delta T)(X,Y),
\end{split}
\end{equation}
where $\delta=(-1)^{np+n+1}*d*$ is the co-differential acting on $p$-forms and $*$ is the Hodge star operator satisfying $*^2=(-1)^{p(n-p)}$.


\subsection{Torsion connection preserving an almost Hermitian structure}In this section we recall the notions and existence of a metric linear connection preserving a given almost Hermitian structure and having totally skew-symmetric torsion from \cite{FI}.

Let $(M,g,J)$ be an almost Hermitian 2n-dimensional manifold with Riemannian
metric $g$ and an almost complex structure $J$ satisfying
$$g(JX,JY)=g(X,Y).$$
The Nijenhuis tensor $N$, the K\"ahler
form $F$ and the Lee form $\theta$ are defined by
\begin{equation}\label{cy1}
N=[J.,J.]-[.,.]-J[J.,.]-[.,J.], \quad F=g(.,J.), \quad \theta(.)=\delta F(J.),
\end{equation}
respectively, where $\delta=-*d*$ is the co-differential on an even dimensional manifold.

Consequently, the  1-form $J\theta$ defined by $J\theta(X)=-\theta(JX)$ is co-closed due to $\delta(J\theta)=-\delta^2 F=0$.

\noindent The Nijenhuis tensor of type (0,3),  denoted by the same letter, is defined by $N(X,Y,Z)=g(N(X,Y),Z)$.

In general, on an almost Hermitian manifold, there not always exists a metric connection with totally skew-symmetric torsion preserving the almost Hermitian structure.
 It is shown in \cite[Theorem~10.1]{FI} that there
exists a unique linear connection $\sb$ preserving an almost hermitian
structure $(g,J)$, $$\sb g=\sb J=0$$ and having totally skew-symmetric torsion $T$  if and only if
the Nijenhuis tensor $N$ is a 3-form.  This is precisely the class $G_1$ in the Gray-Hervella classification of almost Hermitian manifolds \cite{GrH} and includes Hermitian manifolds, $N=0$ as well as nearly K\"ahler spaces, $(\LC_XJ)X=0$. We call this connection \emph{the torsion connection}.

The torsion $T$ of the torsion connection is
determined by \cite[Theorem~10.1]{FI}
\begin{equation}\label{cy2}
 T=JdF+N=-dF(J.,J.,J.)+N=-dF^+(J.,J.,J.)+\frac{1}{4}N,
\end{equation}
where $dF^+$ denotes the (1,2)+(2,1)-part of $dF$ with respect to $J$ and satisfies the equality
\begin{equation}\label{12}
\begin{split}
dF^+(X,Y,Z)=dF^+(JXJ,Y,Z)+dF^+(JX,Y,JZ)+dF^+(X,JY,JZ).
\end{split}
\end{equation}
It follows from \eqref{12} that
  \begin{equation}\label{normH}
dF^+(Je_k,Je_i,e_j)dF^+(e_k,e_i,e_j)=\frac13||dF^+||^2.
\end{equation}
The (3,0)+(0,3)-part $dF^-$ obeys the identity
$$dF^-(JX,Y,Z)=dF^-(X,JY,Z)=dF^-(X,Y,JZ)$$
and  is determined
completely by the Nijenhuis tensor \cite{gauduchon}.  If $N$ is a three form  then (see e.g.\cite{gauduchon,FI})
\begin{equation*}
dF^-(X,Y,Z)=\frac{3}{4}N(JX,Y,Z).
\end{equation*}
The Lee form $\theta:=\delta F\circ J$ of the $G_1$- manifold is given in terms of the  torsion $T$ as
 \begin{equation}\label{liff}
\theta(X)=-\frac12T(JX,e_i,Je_i)=\frac12g(T(JX),F)=\frac12(F\lrcorner
T(JX))~,~~~\theta_i={1\over 2} J^k{}_i T_{kj\ell} F_{j\ell}~,
\end{equation}
where $\{e_1,\dots,e_n,e_{n+1}=Je_1,\dots,e_{2n}=Je_n\}$ is an orthonormal basis.

When the almost complex structure is integrable, $N=0$, the torsion connection was defined and studied by Strominger \cite{Str} in connection with the heterotic string background and was used by Bismut to prove a local index theorem for
the Dolbeault operator on Hermitian non-K\"ahler manifold \cite{bismut}. This formula was applied also in string theory (see e.g. \cite{BBE}). In the Hermitian case, the torsion connection is also
known as the Strominger-Bismut connection and has attained a lot of consideration both in mathematics and physics (see the Introduction).

In the nearly K\"ahler case,  the torsion connection $\sb$ is the characteristic connection
considered by Gray, \cite{gray} and the torsion $T$ and the Nijenhuis tensor are $\sb-$parallel, $\sb T=\sb N = 0$ (see e.g. \cite{Kir,BM}).

 In general $G_1$ manifold,  the holonomy group of $\sb$ is contained in $U(n)$. The reduction of the holonomy group of the torsion connection to a subgroup of $SU(n)$ can be expressed in terms of its Ricci 2-form which represents the first Chern class, namely
$$\rho(X,Y)=\frac12R(X,Y,e_i,Je_i)=0.
$$
We remark that the formula  \cite[(3.16)]{IP2} holds in the general case of a $G_1$-manifold (see also \cite{FI}), namely, the next formula holds true
\begin{equation}\label{ssc}
\rho(X,Y)=Ric(X,JY)+(\sb_X\theta)(JY)+\frac14dT(X,Y,e_i,Je_i).
\end{equation}
Thus, on a $G_1$ manifold the restricted holonomy group of the torsion connection is contained in $SU(n)$ if and only if the next  condition holds \cite[Theorem~10.5]{FI}
\begin{equation}\label{su}
Ric(X,Y)+(\sb_X\theta)(Y)-\frac14dT(X,JY,e_i,Je_e)=0.
\end{equation}
We calculate from \eqref{dh} using \eqref{cy2} and \eqref{normH}
\begin{multline}\label{su2}
dT(e_j,Je_j,e_i,Je_i)=-8(\sb_{e_i}\theta)(e_i)+8|\theta|^2-4T(e_i,e_j,e_k)T(Je_i,Je_j,e_k)\\
=8\delta \theta+8||\theta||^2-\frac43||dF^+||^2+\frac14||N||^2=8\delta \theta+8|\theta|^2-\frac43||T||^2+\frac13||N||^2.
\end{multline}
\begin{rmrk}
It follows from \eqref{su2} that if the structure is balanced, $\theta=0$, and the torsion is closed, $dT=0$ then $16|dF|^2=3|N|^2$. In the complex case, $N=0$,  this confirms the old result first established in \cite{AI}  that  a balanced Hermitian manifold with closed torsion must be K\"ahler.
\end{rmrk}
The trace in \eqref{su} together with \eqref{su2} and \eqref{rics} yield the following formulas for the scalar curvature of the torsion connection and the Riemannian scalar curvature on an ACYT space
\begin{equation}\label{scal2}
\begin{split}
Scal=3\delta\theta+2||\theta||^2-\frac13||dF^+||^2+\frac1{16}||N||^2;\qquad
Scal^g=3\delta\theta+2||\theta||^2-\frac13||dF^+||^2+\frac5{64}||N||^2.
\end{split}
\end{equation}
We obtain from \eqref{ssc}, \eqref{scal2} and \eqref{su2}  the following
\begin{thrm}\label{thconj2}
Let $(M,g,J)$ be 
a 2n-dimensional balanced ACYT space.
\begin{itemize}
\item[a)] If either the Riemannian scalar curvature or the scalar curvature of the torsion connection vanishes, $s^g=0$ or $s=0$ then
\begin{equation}\label{sdt}
16||dF^+||^2=3||N||^2.
\end{equation}
\item[b)] If the Ricci tensor of the torsion connection vanishes, $Ric=0$ then
$dT(X,Y,e_i,Je_i)=0.$

In particular, \eqref{sdt} holds.
\end{itemize}
\end{thrm}
Consequently, we have
\begin{cor}
On a balanced CYT manifold if either the Riemannian scalar curvature or the scalar curvature of the Strominger-Bismut connection vanishes, $Scal^g=0$ or $Scal=0$ then it must be a K\"ahler Ricci flat, i.e. a Calabi-Yau space.

In particular, any balanced CYT space with vanishing Ricci tensor of the Strominger-Bismut connection is K\"ahler Ricci flat, i.e. it is a Calabi-Yau space.
\end{cor}

 Following \cite{AOUV} we have
 \begin{dfn}[\cite{AOUV}]\label{dkel}
 We say that a curvature tensor $R$ is   K\"ahler like if it resembles the K\"ahler curvature  identities, i.e. in addition to \eqref{r1} and \eqref{RB} it satisfies
 \begin{equation}\label{kelprop}
 \begin{split}
 R(X,Y,JZ,JV)=R(X,Y,Z,V);
 \end{split}
 \end{equation}
  \end{dfn}
 Since the torsion connection on an almost Hermitian manifold with skew-symmetric Nijenhuis tensor preserves the Hermitian metric and the complex structure its curvature satisfies \eqref{r1} and \eqref{kelprop}. Clearly,  the curvature of the torsion connection is of K\"ahler type exactly when it satisfies the Riemannian first Bianchi identity \eqref{RB}. In this case, the Ricci tensor is symmetric, $J$-invariant and $Ric(X,JY)=\rho(X,Y)$.

\section{SU(3)-structures}

Let $(M,g,J)$ be an almost Hermitian 6-manifold with a Riemannian
metric $g$ and an almost complex structure $J$, i.e. $(g,J)$
define an $U(3)$-structure.

An $SU(3)$-structure is determined by an additional  non-degenerate
(3,0)-form $\Psi=\Psi^++\sqrt{-1}\Psi^-$ of constant norm, or equivalently by a
non-trivial spinor. To be more explicit, we may choose a local
orthonormal frame  $e_1,\dots,e_6$,  identifying it  with the dual
basis via the metric. Write $e_{i_1 i_2\dots i_p}$ for the
monomial $e_{i_1} \wedge e_{i_2} \wedge \dots \wedge e_{i_p}$. An
$SU(3)$-structure is described locally by
\begin{equation}\label{AA}
\begin{split}
\Psi=-(e_1+\sqrt{-1}e_2)\wedge (e_3+\sqrt{-1}e_4)\wedge  (e_5+\sqrt{-1}e_6),\quad
F =-e_{12} - e_{34}- e_{56}, \\
\Psi^+ =-e_{135} + e_{236} + e_{146} +e_{245}, \quad
  \Psi^- =-e_{136} - e_{145} - e_{235} + e_{246}.
\end{split}
\end{equation}
These forms are subject to the compatibility relations
\begin{equation*}
F\wedge\Psi^{\pm}=0, \qquad \Psi^+\wedge\Psi^-=-\frac23 F^3;
\end{equation*}
The subgroup of $SO(6)$ fixing the forms $F$ and  $\Psi$ simultaneously is $SU(3)$.
The two forms $F$ and $\Psi$ determine the metric completely. The Lie algebra of $SU(3)$
is denoted $su(3)$.

The failure of the holonomy group of the Levi-Civita connection $\sb^g$ of the metric $g$ to be reduced to $SU(3)$
can be measured by the intrinsic torsion $\tau$, which is identified with $dF, d\ps$ and $d\sp$
and can be decomposed into five classes \cite{CSal}, $\tau \in W_1\oplus \dots \oplus W_5$. The
intrinsic torsion of an $U(n)$- structure belongs to the first four components described
by Gray-Hervella \cite{GrH}. The five components of a $SU(3)$-structure are first described
by Chiossi-Salamon \cite{CSal} (for interpretation in physics see \cite{Car,GLMW,GM}) and are
determined by $dF,d\Psi^+,d\Psi^-$ as well as $dF$ and $N$. We list below those of these classes which
we will use later.
\begin{description}
\item[$\tau \in W_1:$] The class of Nearly K\"ahler (weak holonomy) manifold defined by
$dF$ to be (3,0)+(0,3)-form.
\item[$\tau \in W_3:$] The class of balanced Hermitian manifold determined by the
conditions $N=\theta=0$.
\item[$\tau \in W_1\oplus W_3\oplus W_4:$] The class called by Gray-Hervella $G_1$-manifolds
determined by the condition that the Nijenhuis tensor is totally skew-symmetric.
This is the precise class which we are interested in.
\item[$\tau \in W_1\oplus W_3:$] The class of balanced  $G_1$-manifolds
determined by the condition that the Nijenhuis tensor is totally skew-symmetric and the Lee form vanishes, $N(X,Y,Z)=-N(X,Z,Y),\quad \theta=0$.
\item[$\tau \in W_3\oplus W_4:$] The class of Hermitian manifolds, $N=0$.
\end{description}
If all five components are zero then we have a Ricci-flat K\"ahler (Calabi-Yau) 3-fold.

Let $(M,g,J,\Psi)$ be an $SU(3)$ manifold. Letting the four form $$\Phi=\frac12F^2=\frac12F\wedge F$$ we have the relations
\begin{equation}\label{star}
\begin{split}
*\Phi=-F,\quad *F=-\Phi;\qquad
*\Psi^+=\Psi^-,\quad *\Psi^-=-\Psi^+\\
*(\alpha\wedge F)=-\alpha\lrcorner\Phi, \quad \alpha\in \Lambda^1,\quad
*(\alpha\wedge\Phi)=J\alpha=-\alpha\lrcorner F, \quad \alpha\in\Lambda^1,\\
*(\alpha\wedge\Psi^+)=-\alpha\lrcorner\Psi^- \quad \alpha\in \Lambda^1,\qquad
*(\alpha\wedge\Psi^-)=\alpha\lrcorner\Psi^+ \quad \alpha\in \Lambda^1,\\
*(\beta\wedge\ps)=\beta\lrcorner\sp, \quad \beta\in \Lambda^2,\qquad
*(\beta\wedge\sp)=-\beta\lrcorner\ps, \quad \beta\in \Lambda^2
\end{split}
\end{equation}
Extending the action of the almost complex structure on the exterior algebra, $J\beta=(-1)^p\beta(J.,\dots,J.)$ for a p-form $\beta\in\Lambda^p$, we recall the decomposition of the exterior algebra under the action of $SU(3)$ from \cite{BV} (see also \cite{LLR}). Let $V=T_pM$ be the vector space at a point $p\in M$ and $V^*$ be the dual vector space.  Obviously $SU(3)$
acts irreducibly on $V^*$ and $\Lambda^5V^*$, while $\Lambda^2V^*$ and $\Lambda^3V^*$ decompose as follows:
\begin{equation*}\label{dec}
\begin{split}
\Lambda^2V^*=\Lambda^2_1V^*\oplus\Lambda^2_6V^*\oplus\Lambda^2_8V^*;\\
\Lambda^3V^*=\Lambda^3_{Re}V^*\oplus\Lambda^3_{Im}\oplus\Lambda^3_6V^*\oplus\Lambda^3_{12}V^*,
\end{split}
\end{equation*}
where
\begin{equation*}
\begin{split}
\Lambda^2_1V^*=\mathbb R.F\quad
\Lambda^2_6V^*=\{\phi\in \Lambda^2V^* | J\phi =-\phi\}\\
\Lambda^2_8V^*=\{\phi\in \Lambda^2V^* | J\phi=\phi, \phi\wedge F^2=0\}\\
\Lambda^3_{Re}V^*=\mathbb R\Psi^+; \qquad \Lambda^3_{Im}=\mathbb R\Psi^-;\\
\Lambda^3_6V^*=\{\alpha\wedge F  |  \alpha\in\Lambda^1V^*\};\\
\Lambda^3_{12}V^*=\{\gamma\in \Lambda^3V^* | \gamma\wedge F=\gamma\wedge\Psi^+=\gamma\wedge\Psi^-=0\}.
\end{split}
\end{equation*}
Denote by $S^2_-$ the space of symmetric 2-tensors which are  of type (2,0)+(0,2) with respect to the almost complex structure $J$, i.e.
\begin{equation}\label{2s}
h(X,Y)=h(Y,X),\qquad h(JX,JY)=-h(X,Y).
\end{equation}
It is known (see \cite{BV}) that the map $\gamma: S^2_-\leftrightarrow \Lambda^3_{12}$ defined by
\begin{equation}\label{s12}
\gamma(h_{ij})=h_{ip}\ps_{pjk}+h_{jp}\ps_{pki}+h_{kp}\ps_{pij},\qquad  h_{im}=\gamma^{-1}(B_{ijk})=\frac14B_{ijk}\ps_{mjk}
\end{equation}
is an isomorphism of $su(3)$ representations.

Writing
$$
F=\frac12F_{ij}e_{ij},\quad \Psi^+=\frac16\Psi^+_{ijk}e_{ijk}, \quad  \Psi^-=\frac16\Psi^-_{ijk}e_{ijk},\quad  \Phi=\frac1{24}\Phi_{ijkl}e_{ijkl}
$$
we have the obvious identites (c.f. \cite{BV})
\begin{equation}\label{iden}
\begin{split}
\Phi_{jslm}=F_{js}F_{lm}+F_{sl}F_{jm}+F_{lj}F_{sm},\\
\Psi^+_{ipq}F_{pq}=0=\Psi^-_{ipq}F_{pq}, \quad \Phi_{ijkl}\Psi^+_{jkl}=0=\Phi_{ijkl}\Psi^-_{jkl},\\
F_{ip}F_{pj}=-\delta_{ij},\quad \Psi^+_{ijs}F_{sk}=-\Psi^-_{ijk}, \quad \Psi^-_{ijs}F_{sk}=\Psi^+_{ijk},\\
\Psi^+_{ipq}\Psi^-_{jpq}=-4F_{ij}, \quad \Psi^+_{ipq}\Psi^+_{jpq}=4\delta_{ij}=\Psi^-_{ipq}\Psi^-_{jpq},\\
\Psi^+_{kls}\Psi^-_{ijs}=\delta_{kj}F_{li}+\delta_{li}F_{kj}-\delta_{ki}F_{lj}-\delta_{lj}F_{ki},\\
\Psi^+_{kls}\Psi^+_{ijs}=F_{kj}F_{li}-\delta_{li}\delta_{kj}-F_{ki}F_{lj}+\delta_{lj}\delta_{ki}\\
 \Phi_{ijkl}F_{kl}=4F_{ij},\quad \Phi_{ijkl}\Psi^+_{klp}=2\Psi^+_{ijp},\quad  \Phi_{ijkl}\Psi^-_{klp}=2\Psi^-_{ijp},\\
 \Phi_{ijkl}\Psi^+_{lqp}=-F_{ij}\Psi^-_{pqk}-F_{jk}\sp_{pqi}-F_{ki}\sp_{pqj},\\
\Phi_{ijkl}\Psi^-_{lqp}=F_{ij}\Psi^+_{pqk}+F_{jk}\ps_{pqi}+F_{ki}\ps_{pqj},\\
\Phi_{ijkl}\Phi_{rjkl}=12\delta_{ir}, \quad  \Phi_{ijkl}\Phi_{klqr}=2F_{ij}F_{qr}-2\delta_{jq}\delta_{ir}+2\delta_{iq}\delta_{jr},\\
\end{split}
\end{equation}
Now, we prove the following algebraic fact which is crucial for our investigations.
\begin{prop}\label{4-for}
Let $A$ be a 4-form and define the 3-form $B=(X\lrcorner A)$ for any $X\in T_pM$. If  the 3-form $B\in\Lambda^3_{12}$ then the four form $A$ vanishes identically, $A=0$
\end{prop}
\begin{proof}
According to \eqref{s12} and \eqref{2s} the tensor
\begin{multline}\label{c1}
C(X,Y,Z)=4\gamma^{-1}(X\lrcorner A)(Y,Z)=A(X,Y,e_i,e_j)\ps(Z,e_i,e_j)\\=C(X,Z,Y)=-C(Z,X,Y)=-C(Z,Y,X)=C(Y,Z,X)=C(Y,X,Z)=-C(X,Y,Z)
\end{multline}
and therefore vanishes. We  get from \eqref{c1} using  \eqref{iden}
\begin{multline}\label{ll1}
0=A_{pqkl}\ps_{kls}\ps_{ijs}=A_{pqkl}\Big[F_{kj}F_{li}-\delta_{li}\delta_{kj}-F_{ki}F_{lj}+\delta_{lj}\delta_{ki}\Big]=2A_{pqkl}F_{kj}F_{li}-2A_{pqji}
\end{multline}
The equalities \eqref{c1} and \eqref{ll1} yield
\begin{multline*}
A(X,Y,JZ,JV)=A(X,Y,Z,V)=A(X,JY,JZ,V)=A(X,JY,Z,JV)=-A(X,Y,JZ,JV)
\end{multline*}
Hence $A=0$ which completes the proof of Proposition~\ref{4-for}.
\end{proof}

\section{Almost Calabi-Yau with torsion space  in dimension 6}
The presence of a parallel spinor with respect to a metric connection with torsion 3-form  in dimension 6 leads
 to the reduction to $SU(3)$, i.e. 
 the existence of an almost
Hermitian structure and a linear
connection preserving the almost Hermitian structure with torsion
3-form 
with holonomy group inside 
SU(3).

The reduction of the holonomy group of the torsion connection to $SU(3)$, was investigated intensively in the complex case, i.e. when the induced almost complex structure is integrable ($N=0$) and it is known as the Strominger condition. The Strominger condition (see \cite{Str}) has different very useful expressions in terms of the five torsion classes (see e.g. \cite{Car,LY,yau,yau1, II}). We present here the result from \cite{II} which gives  necessary and sufficient condition to have a $\nabla-$parallel spinor in dimension 6,
\begin{thrm}\cite[Theorem~4.1]{II}\label{cythm1}
Let $(M,g,J,\Psi)$ be a 6-dimensional smooth manifold with an $SU(3)$-structure $(g,J,\Psi)$.
The next two conditions are equivalent:
\begin{enumerate}
\item[a)] $(M,g,J,\Psi)$ is an ACYT space, i.e. there exists a unique $SU(3)$-connection $\sb$ with torsion 3-form 
whose  holonomy is contained in SU(3),
 $\sb g=\sb J=\sb \Psi^+=\sb\Psi^-=0$;
\item[b)] The Nijenhuis tensor $N$ is totally skew-symmetric and the following two conditions
hold
\begin{equation}\label{cycon}
\begin{split}
d\Psi^+=\theta\wedge\Psi^+ -\frac{1}{4}(N,\Psi^+)* F; \qquad
d\Psi^-=\theta\wedge\Psi^- -\frac{1}{4}(N,\Psi^-)* F.
\end{split}
\end{equation}
 The torsion is given by
\begin{equation}\label{torcy}
\begin{split}
T=-*dF +*(\theta\wedge F)+\frac{1}{4}(N,\Psi^+)\Psi^+ + \frac{1}{4}(N,\Psi^-)\Psi^-.
\end{split}
\end{equation}
\end{enumerate}
\end{thrm}
Our first observation is 
\begin{prop}\label{tnnew}
Let $(M,g,J,\Psi)$ be an   ACYT  6-manifold.
The following conditions are equivalent:
\begin{itemize}
\item[a)] The exterior derivative of the Lee form is $J$-invariant, $ d\theta(X,Y)=d\theta(JX,JY);$
\item[b)] The Nijenhuis tensor is of a constant norm, $||N||^2=const.$;
\item[c)] The Nijenhuis tensor is parallel with respect to the torsion connection, $\sb N=0$.
\end{itemize}
In particular, on a 6-dimensional CYT manifold, the 2-form $d\theta$ is of type (1,1).
\end{prop}
\begin{proof}
For the torsion  $T$, we obtain  from \eqref{torcy} using \eqref{star} that
\begin{equation}\label{torcy1}
\begin{split}
T=-*dF +*(\theta\wedge F)+\frac{1}{4}(N,\Psi^+)\Psi^+ + \frac{1}{4}(N,\Psi^-)\Psi^-.\\
=*d*\Phi-\theta\lrcorner\Phi+\frac{1}{4}(N,\Psi^+)\Psi^+ + \frac{1}{4}(N,\Psi^-)\Psi^-.\\
=-\delta\Phi-\theta\lrcorner\Phi+\frac{1}{4}(N,\Psi^+)\Psi^+ + \frac{1}{4}(N,\Psi^-)\Psi^-
\end{split}
\end{equation}
The formula for the codifferential of a 4-form in terms of the Levi-Civita connection $\sb^g$ and \eqref{tsym} imply
\begin{multline}\label{dphi}
-\delta\Phi_{klm}=\sb^g_j\Phi_{jklm}=\sb_j\Phi_{jklm}-\frac12T_{jsk}\Phi_{jslm}+\frac12T_{jsl}\Phi_{jskm}-\frac12T_{jsm}\Phi_{jskl}\\
=-\frac12T_{jsk}\Phi_{jslm}+\frac12T_{jsl}\Phi_{jskm}-\frac12T_{jsm}\Phi_{jskl},
\end{multline}
since $\sb\Phi=0$.

Substitute \eqref{dphi} into \eqref{torcy1} to get the  expression for the 3-form torsion $T$,
\begin{equation}\label{torcy2}
T_{klm}=-\frac12T_{jsk}\Phi_{jslm}+\frac12T_{jsl}\Phi_{jskm}-\frac12T_{jsm}\Phi_{jskl}-\theta_s\Phi_{sklm}+\lambda\ps_{klm} +\mu\sp_{klm},
\end{equation}
where the functions $\lambda,\mu$ are defined by the  scalar products
\begin{equation}\label{lm}
\lambda =\frac14(N,\ps)=\frac1{24}N_{ijk}\ps_{ijk}, \qquad \mu=\frac14(N,\sp)=\frac1{24}N_{ijk}\sp_{ijk}.
\end{equation}
It is straightforward to check that the Lee form $\theta$ can be expressed in terms of the torsion $T$, the 2-form $F$ and  the 4-form $\Phi$ as follows
\begin{equation}\label{tit}
\theta_i=\frac12T_{sjk}F_{jk}F_{si}=\frac16T_{jkl}\Phi_{jkli}; \quad \theta_iF_{iq}=-\frac12T_{qjk}F_{jk}. 
\end{equation}
Further, we calculate from \eqref{torcy2} the functions $\lambda$ and $\mu$ in terms of the torsion $T$ using \eqref{iden} as follows
\begin{equation}\label{lmt}
\begin{split}
T_{klm}\ps_{klm}=-\frac32T_{jsk}\Phi_{jslm}\ps_{klm}+24\lambda=-3T_{jsk}\ps_{jsk}+24\lambda \Longrightarrow \lambda=\frac16T_{klm}\ps_{klm};\\
T_{klm}\sp_{klm}=-\frac32T_{jsk}\Phi_{jslm}\sp_{klm}+24\mu=-3T_{jsk}\sp_{jsk}+24\mu  \Longrightarrow \mu=\frac16T_{klm}\sp_{klm}.
\end{split}
\end{equation}
We  have from \eqref{torcy2} using the first  formula in \eqref{iden} and \eqref{tit} that 
\begin{equation}\label{torcy21}
\begin{split}
T_{klm}=-\frac12T_{jsk}\Phi_{jslm}+\frac12T_{jsl}\Phi_{jskm}-\frac12T_{jsm}\Phi_{jskl}-\theta_s\Phi_{sklm}+\lambda\ps_{klm} +\mu\sp_{klm}\\=
\theta_sF_{sk}F_{lm}+T_{jsk}F_{jm}F_{ls}-\theta_sF_{sl}F_{km}-T_{jsl}F_{jm}F_{ks}-\theta_sF_{sm}F_{lk}-T_{jsm}F_{jk}F_{ls}\\-\theta_s[F_{sk}F_{lm}+F_{kl}F_{sm}+F_{ls}F_{km}]+\lambda\ps_{klm} +\mu\sp_{klm}\\=
T_{jsk}F_{jm}F_{ls}-T_{jsl}F_{jm}F_{ks}-T_{jsm}F_{jk}F_{ls}+\lambda\ps_{klm} +\mu\sp_{klm}.
\end{split}
\end{equation}
We can write the formula \eqref{torcy21} in the form
\begin{equation}\label{nijcy}
T(X,Y,Z)=T(JX,JY,Z)+T(JX,Y,JZ)+T(X,JY,JZ)+\lambda\ps(X,Y,Z) +\mu\sp(X,Y,Z).
\end{equation}
Hence, we get from \eqref{cy2} and \eqref{nijcy}
\begin{equation}\label{nn1}
\begin{split}
N=\lambda\ps+\mu\sp, \quad ||N||^2=24(\lambda^2+\mu^2).
\end{split}
\end{equation}
\begin{lemma}\label{lcom}
On a 6-dimensional ACYT manifold, the next identities hold
\begin{gather}\label{new2}
 d\mu=Jd\lambda, \quad d\lambda=-Jd\mu,\\
\label{new3}
d\theta(X,Y)-d\theta(JX,JY)=-(d\mu\lrcorner\sp)(X,Y)
=-(d\lambda\lrcorner\ps)(X,Y).
\end{gather}
\end{lemma}
\begin{proof}
Using  \eqref{iden}, we obtain from \eqref{cycon}
\begin{equation}\label{cycon1}
\begin{split}
0=d^2\ps=d\theta\wedge\ps+\lambda\theta\wedge *F-d\lambda\wedge*F-\lambda d*F=d\theta\wedge\ps-\lambda\theta\wedge\Phi-\lambda *J\theta-d\lambda\wedge*F\\=d\theta\wedge\ps-d\lambda\wedge*F,\\
0=d^2\sp=d\theta\wedge\sp+\mu\theta\wedge *F-\mu d*F-d\mu\wedge*F=d\theta\wedge\sp-\mu\theta\wedge \Phi-\mu*J\theta-d\mu\wedge*F\\=d\theta\wedge\sp-d\mu\wedge*F,
\end{split}
\end{equation}
where we used the identities
 \[
J\theta=\delta F=-*d*F=*(dF\wedge F),\quad *J\theta=-dF\wedge F=d*F.
\]
The  equalities  \eqref{cycon1} and \eqref{star} imply
\begin{equation}\label{new5}
d\theta_{ij}\sp_{ijk}=2d\lambda_sF_{sk};\qquad d\theta_{ij}\ps_{ijk}=-2d\mu_sF_{sk}.
\end{equation}
Multiply \eqref{new5} with $F_{kl}$, use \eqref{iden} and \eqref{new5} to obtain
\begin{equation*}
-2d\mu_sF_{sl}=d\theta_{ij}\ps_{ijl}=d\theta_{ij}\sp_{ijk}F_{kl}=-2d\lambda_l,
\end{equation*}
which shows  the identities \eqref{new2} hold.

The  equalities  \eqref{new5} yield applying \eqref{star} and \eqref{iden}
\begin{multline*}
0=\frac12d\theta_{ij}\ps_{ijk}\ps_{kab}+d\mu_sF_{sk}\ps_{kab}=\frac12d\theta_{ij}\Big[F_{ib}F_{ja}-F_{ia}F_{jb}-\delta_{ib}\delta_{ja}+\delta_{ia}\delta_{jb}  \Big]+d\mu_s\sp_{sab}\\=-d\theta_{ij}F_{ia}F_{jb}+d\theta_{ab}+d\mu_s\sp_{sab}.
\end{multline*}
This proves the first equality in \eqref{new3}. The second equality in \eqref{new3} follows from the already established \eqref{new2} and \eqref{star}. The lemma is proved.
\end{proof}
The equivalence between a) and c) is a consequence from \eqref{new3} and \eqref{nn1} since $\sb\ps=\sb\sp=0$.

Clearly, c) implies b). For the converse, the condition $||N||^2=const.$ and \eqref{nn1} imply
\[\lambda^2+\mu^2=const., \lambda d\lambda+\mu d\mu=0\]
which yields applying \eqref{new2}
\[-\lambda Jd\mu+\mu Jd\lambda=0,\quad \lambda d\mu-\mu d\lambda=0.\]
Thus we obtain a homogeneous linear system with respect to $d\lambda$ and $d\mu$.
The determinant of this system is the constant $\lambda^2+\mu^2$. If this constant is zero then $N=0$ and \eqref{nn1} gives $\lambda=\mu=0$. If the constant is not zero one concludes $d\lambda=d\mu=0$. Therefore $\sb N=0$ due to \eqref{nn1}.
\end{proof}
Applying the Hopf's maximum principle we obtain
\begin{thrm}\label{thnnew}
Let $(M,g,J,\Psi=\Psi^++\sqrt{-1}\Psi^-)$ be a 6-dimensional compact  ACYT manifold.

Then the  Nijenhuis tensor is parallel with respect to the torsion connection and  the exterior derivative of the Lee form is $J$-invariant,
$$\sb N=0, \qquad d\theta(X,Y)=d\theta(JX,JY).$$
In particular, the Nijenhuis tensor is of a constant norm, $||N||^2=const.$
\end{thrm}
\begin{proof}
From Lemma~\ref{lcom}, we have $0=d^2\lambda=-dJd\mu$. The trace yields $(dJd\mu)_{ij}F_{ij}=0$ which is an elliptic second-order operator annihilating constants and we apply the Hopf's maximum principle (see e.g. \cite{YB,GFS}) to conclude $\mu=const$. Hence $d\lambda=d\mu=0$ due to \eqref{new2}. Now, \eqref{new3} and Proposition~\ref{tnnew} complete the proof.
\end{proof}
\begin{cor}\label{pnew}
Let $(M,g,J,\Psi)$ be a 6-dimensional  ACYT manifold. 

If the Lee form is closed, $d\theta=0$, then the Nijenhuis tensor is $\sb-$parallel, $\sb N=0$.

If the  Lee form vanishes, $\theta=0$ then the Nijenhuis tensor $N$ is parallel with respect to the torsion connection, the torsion 3-form is co-closed  and the Ricci tensor of $\sb$ is symmetric,
\[\sb N=\delta T=0,\qquad Ric(X,Y)=Ric(Y,X).\]
\end{cor}
\begin{proof}
If $d\theta=0$, one has $\sb N=0$ due to Proposition~\ref{tnnew}.

Suppose $\theta=0$. Then $d\lambda=d\mu=0$ due to \eqref{new3} and \eqref{torcy} together with \eqref{cycon} and \eqref{star} imply
\[d*T=\lambda d\sp-\mu d\ps=(-\lambda\mu+\mu\lambda)\Phi=0.
\]
This completes the proof in view of \eqref{rics}.
\end{proof}
As a consequence of the proof of Proposition~\ref{tnnew} and Corollary~\ref{pnew} we characterize Ricci-flat balanced ACYT.
\begin{thrm}\label{thnew}
Let $(M,g,J,\Psi)$ be a 6-dimensional balanced ACYT manifold,  
$\theta=0$.

The following conditions are equivalent:
\begin{itemize}
\item[a)] The Ricci tensor of the torsion connection vanishes, $Ric=0$;
\item[b)] The torsion 3-form is harmonic, i.e. it is closed and co-closed, $dT=\delta T=0$.
\end{itemize}
\end{thrm}
\begin{proof}
Let the assumption a) holds.
The condition $Ric=0$ implies $\delta T=0$.

On an $SU(3)$ manifold in dimension six, we can express the Ricci tensor in a more appropriate form as follows. Since the torsion connection $\sb$ preserves the $SU(3)$-structure its curvature  lies in the Lie algebra $su(3)$, i.e. it satisfies
\begin{equation}\label{rr}
\begin{split}
R_{ijab}\ps_{abc}=R_{ijab}\sp_{abc}=R_{ijab}F_{ab}=0 \Longleftrightarrow R_{ijab}\Phi_{abkl}=-2R_{ijkl}.
\end{split}
\end{equation}
We have from \eqref{rr} using \eqref{1bi1}, \eqref{tit} and \eqref{dh}  that the Ricci tensor $Ric$ of  $\sb$ is given by
\begin{multline}\label{ricdt}
Ric_{ij}=\frac12R_{iabc}\Phi_{jabc}=\frac16\Big[R_{iabc}+R_{ibca}+R_{icab} \Big]\Phi_{jabc}=\frac1{12}dT_{iabc}\Phi_{jabc}+\frac16\sb_iT_{abc}\Phi_{jabc}\\
=\frac1{12}dT_{iabc}\Phi_{jabc}-\sb_i\theta_j.
\end{multline}
Similarly, we have
\begin{equation}\label{ricnew}
\begin{split}0=R_{iabc}\ps_{abc}=\frac13\Big[R_{iabc}+R_{ibca}+R_{icab} \Big]\ps_{abc}=\frac16dT_{iabc}\ps_{abc}+\frac13\sb_iT_{abc}\ps_{abc},\\
0=R_{iabc}\sp_{abc}=\frac13\Big[R_{iabc}+R_{ibca}+R_{icab} \Big]\sp_{abc}=\frac16dT_{iabc}\sp_{abc}+\frac13\sb_iT_{abc}\sp_{abc}.\end{split}
\end{equation}
From Proposition~\ref{pnew} we know that the condition $\theta=0$ implies $d\lambda=d\mu=0$. Then $Ric=0$, \eqref{ricdt} and \eqref{ricnew} yield
\begin{equation}\label{nthsn}
\begin{split}
dT_{pjkl}\Phi_{jkli}=12\sb_p\theta_i=0,\qquad
dT_{pjkl}\ps_{jkl}=-12\sb_p\lambda=0,\qquad
dT_{pjkl}\sp_{jkl}=-12\sb_p\mu=0.
\end{split}
\end{equation}
The identities \eqref{nthsn} show that the 3-form $(X\lrcorner dT)\in \Lambda^3_{12}$. Then the 4-form $dT=0$ due to Proposition~\ref{4-for}.

For the converse, if we assume b) then \eqref{ricdt} and $\theta=0$  show  $Ric=0$ which completes the proof.
\end{proof}
One gets  from \eqref{ricdt} and Theorem~\ref{tFBI}
\begin{thrm}\label{mainsu}
Let $(M,g,J,\Psi)$ be a 6-dimensional  ACYT  manifold  and  the torsion connection $\sb$   has curvature $R \in S^2\Lambda^2$ i.e. \eqref{r4} holds.

Then the  covariant derivative of the 1-form $J\theta$ with respect to $\sb$ is skew-symmetric,
\[ (\sb_XJ\theta)Y+(\sb_YJ\theta)X=0 \Longleftrightarrow  (\sb_X\theta)JY=-(\sb_Y\theta)JX.
\]
In particular, the vector field dual to the 1-form $J\theta$ is a  Killing vector field.
\begin{itemize}
\item[a)] If in addition, the Nijenhuis tensor  $N$ is of constant norm, $||N||^2=const.$
then $N$ is $\sb-$parallel, the covariant derivative of the Lee form $\theta$ with respect to $\sb$ is symmetric and $J$-invariant,
\begin{equation}\label{symJtheta}
\sb N=0,\qquad (\sb_X\theta)Y=(\sb_Y\theta)X=(\sb_{JX}\theta)JY.
\end{equation}
\item[b)] If the Lee form vanishes, $\theta=0$, i.e. the manifold is balanced $G_1$ space, $M\in W_1\oplus W_3$, then  
$$\sb T=\sb dF=\sb N=0, \qquad Scal^g=-\frac13||dF^+||^2+\frac5{64}||N||^2=const.$$
\end{itemize}
\end{thrm}
\begin{proof}
We get from \eqref{tit} using the fact proved in \cite[Lemma~3.4]{I} that the curvature of a metric connection $\sb$ with totally skew-symmetric torsion $T$ satisfies \eqref{r4} exactly when $\sb T$ is a 4-form, 
\begin{equation}\label{new}
\begin{split}
\sb_p\theta_iF_{iq}=-\frac12\sb_pT_{qjk}F_{jk}=\frac12\sb_qT_{pjk}F_{jk}=-\sb_q\theta_iF_{ip};\\
(\sb_X\theta)JY=-(\sb_Y\theta)JX;\\
(\sb_XJ\theta)Y=-(\sb_X\theta)JY=(\sb_Y\theta)JX=-(\sb_YJ\theta)X.
\end{split}
\end{equation}
\indent It follows from \eqref{new} that the vector field $J\theta$ dual to the 1-form $J\theta$ is Killing since \eqref{tsym} yields
 $(\LC_X J\theta)Y+(\LC_Y J\theta)X=0$ which proves the first part  of Theorem~\ref{mainsu}.

Take the covariant derivative of the torsion expressed in \eqref{torcy2} with respect to $\sb$ to get
\begin{multline}\label{ntor}
\sb_pT_{klm}=-\frac12\sb_pT_{jsk}\Phi_{jslm}+\frac12\sb_pT_{jsl}\Phi_{jskm}-\frac12\sb_pT_{jsm}\Phi_{jskl}\\-\sb_p\theta_s\Phi_{sklm}+\sb_p\lambda\ps_{klm} +\sb_p\mu\sp_{klm},
\end{multline}
where we used  that all the structure forms $g,F,\ps,\sp,\Phi$ are $\sb-$parallel.

The fact that $\sb T$ is a 4-form, taking the trace in \eqref{ntor} and using \eqref{tit}, yields
\begin{multline}\label{no1}
0=\sb_pT_{plm}=\frac12\sb_pT_{jsl}\Phi_{jspm}-\frac12\sb_pT_{jsm}\Phi_{jspl}-\sb_p\theta_s\Phi_{splm}+\sb_p\lambda\ps_{plm} +\sb_p\mu\sp_{plm}\\
=-3\sb_l\theta_m+3\sb_m\theta_l+\sb_p\lambda\ps_{plm}+\sb_p\mu\sp_{plm}+\frac12(\sb_p\theta_s-\sb_s\theta_p)\Big(F_{ps}F_{lm}+F_{sl}F_{pm}+F_{lp}F_{sm}\Big)\\
=-3\sb_l\theta_m+3\sb_m\theta_l+\sb_p\lambda\ps_{plm}+\sb_p\mu\sp_{plm}+\frac12(\sb_p\theta_s-\sb_s\theta_p)\Big(F_{sl}F_{pm}+F_{lp}F_{sm}\Big)\\
=-3\sb_l\theta_m+3\sb_m\theta_l+\sb_p\lambda\ps_{plm}+\sb_p\mu\sp_{plm}+(\sb_p\theta_s-\sb_s\theta_p)F_{sl}F_{pm},
\end{multline}
where we applied  the general identity $\sb_i\theta_sF_{si}=0$ which is a consequence of the fact that the 1-form $J\theta$ is co-closed. Indeed, we have  $0=\delta(J\theta)=-\LC_i(J\theta)_i=-\sb_i(J\theta)_i=-\sb_i\theta_sF_{si}$ applying \eqref{tsym}.

 Substitute  \eqref{new} into \eqref{no1} to get
 \begin{equation}\label{ew4}
d^{\sb}\theta_{lm}= \sb_l\theta_m-\sb_m\theta_l=\frac12[\sb_p\lambda\ps_{plm}+\sb_p\mu\sp_{plm}]=\sb_p\lambda\ps_{plm}=\sb_p\mu\sp_{plm},
 \end{equation}
 where we used \eqref{new2} to conclude the last equality.

 The condition $||N||^2=const.$ and Proposition~\ref{tnnew}  imply $d\lambda=d\mu=0$ and one gets $d^{\sb}\theta=0$ from \eqref{ew4}. Then \eqref{new} completes the proof of a) in Theorem~\ref{mainsu}.

To show the condition $\theta=0$ yields $\sb T=0$ we observe that \eqref{new3} implies the functions $\lambda$ and  $\mu$ are constants. In particular $\sb N=0$ due to \eqref{nn1}. Then \eqref{lmt} together with  \eqref{tit} imply
\begin{equation}\label{nths}
\begin{split}
\sb_pT_{jkl}\Phi_{jkli}=6\sb_p\theta_i=0,\qquad
\sb_pT_{jkl}\ps_{jkl}=6\sb_p\lambda=0,\qquad
\sb_pT_{jkl}\sp_{jkl}=6\sb_p\mu=0.
\end{split}
\end{equation}
The identities \eqref{nths} show that the 3-form $(X\lrcorner\sb T)\in \Lambda^3_{12}$. Hence, the 4-form $\sb T=0$ due to Proposition~\ref{4-for}. Then the formulas \eqref{cy2} and  \eqref{scal2}  complete the proof of the Theorem~\ref{mainsu}.
\end{proof}
\subsection{Example} We take the next example from \cite{II} which supports Theorem~\ref{mainsu}.

Let $G$ be the six-dimensional connected simply connected and nilpotent
 Lie group, determined by the left-invariant 1-forms $\{e_1,\dots,e_6\}$
 such that
\begin{gather}\label{in1}
de_2=de_3=de_6=0,\\\nonumber
de_1=e_3\wedge e_6,\quad
de_4= e_2\wedge e_6, \quad
de_5= e_2\wedge e_3.
\end{gather}
Consider the metric on $G\cong \mathbb R^6$ defined by
$g=\sum_{i=1}^6e_i^2$.
Let $(F,\Psi)$ be the $SU(3)$-structure on $G$ given by \eqref{AA}. Then$(G,F,\Psi)$ is
an almost complex manifold with a $SU(3)$-structure.

It is easy to verify using  \eqref{cy1} and \eqref{in1} that
\begin{gather}\label{in2}
dF=-3e_{236}, \quad  N=-\Psi^-, \quad d\Psi^-=*F, \quad
(N,\Psi^-)=-4, \\ \nonumber \theta=d\Psi^+=(N,\Psi^+)=0.
\end{gather}

Hence, $(G,\Psi,g,J)$ is neither complex nor Nearly K\"ahler manifold but it fulfills
the conditions \eqref{cycon} of Theorem~\ref{cythm1} and
therefore there exists a $SU(3)$ connection with torsion 3-form on $(G,\Psi,g,J)$.
The expression
\eqref{torcy} and \eqref{in2} give
\begin{equation}\label{tor}
T=-2e_{145}+e_{136}+e_{235}-e_{246}, \quad dT=-2(e_{1256}+e_{3456}+e_{1234})=2*F.
\end{equation}
The equalities  \eqref{tor} and  \eqref{cy2} imply  the nonzero essential terms
of the torsion connection are (c.f. \cite{II})
\begin{gather}\label{tor1}
\nabla_{e_1}e_6=-e_3, \qquad  \nabla_{e_5}e_2= e_3, \qquad \nabla_{e_4}e_6=-e_2, \\\nonumber
\nabla_{e_4}e_5=-e_1, \qquad \nabla_{e_5}e_1=-e_4, \qquad \nabla_{e_1}e_4=-e_5.
\end{gather}
It is easy to verify {cf. \cite{II}) that the curvature of the torsion connection satisfies \eqref{r4} and  the torsion tensor $T$ as well as the Nijenhuis tensor $N$ and $dF$ are parallel with respect to the connection $\nabla$.

The coefficients of the structure equations of the Lie algebra given by \eqref{in1}
are integers. Therefore, the well-known theorem of Malcev \cite{Mal} states that
the group $G$ has a uniform discrete subgroup $\Gamma$ such that $Nil^6=G/\Gamma$ is a
compact 6-dimensional nil-manifold. The $SU(3)$-structure, described above, descends to $Nil^6$
and therefore we obtain a compact example.

\section{Proof of Theorem~\ref{mainsu3} and Theorem~\ref{cmainsu3}}
We begin with the next
\begin{lemma}\label{sbtheta}
Let $(M,g,J,\Psi)$ be an ACYT 6-manifold 
satisfying the condition \eqref{s2l2}. 

 Then the Lee form $\theta$ is $\sb-$parallel, $\sb\theta=0.$

In particular, the Lee form is co-closed, $\delta\theta=0$.
\end{lemma}
\begin{proof}
We already know from  Theorem~\ref{mainsu} a)  that  $d\lambda=d\mu=0$ and \eqref{symJtheta} holds.

Since  $\sb T$ is a 4-form, substitute \eqref{dh} into \eqref{su} to get
\begin{multline}\label{su1}
0=Ric(X,Y)+(\sb_X\theta)(Y)-(\sb_X T)(JY,e_i,Je_e)-\frac12\sigma^T(X,JY,e_i,Je_i)\\
=Ric(X,Y)+(\sb_X\theta)(Y)+2(\sb_X\theta)Y-\frac12\sigma^T(X,JY,e_i,Je_i)\\
=Ric(X,Y)+3(\sb_X\theta)(Y))-\frac12\sigma^T(X,JY,e_i,Je_i).
\end{multline}
Let $Ric=0$. Then we have from \eqref{su}, \eqref{su1} and \eqref{liff}
\begin{equation}\label{li}
\begin{split}
-\frac12(\sb_XT)(JY,e_i,Je_i)=(\nabla_X\theta)Y=\frac14dT(X,JY,e_i,Je_e)=\frac16\sigma^T(X,JY,e_i,Je_i)\\
=-\frac13T(X,JY,J\theta)-\frac13T(X,e_i,e_j)T(JY,Je_i,e_j).
\end{split}
\end{equation}
Using that $d^{\sb}\theta(X,Y)=(\sb_X\theta)Y-(\sb_Y\theta)X=-(d\mu\lrcorner\sp)(X,Y)=0$ due to \eqref{new3} and \eqref{symJtheta}, we get
\begin{equation*}
d\theta(X,Y)=(\sb^g_X\theta)Y-(\sb^g_Y\theta)X=-(d\mu\lrcorner\sp)(X,Y)+T(X,Y,\theta),
\end{equation*}
which combined with \eqref{new3} and \eqref{symJtheta} yields
\begin{equation}\label{dt2}
T(X,Y,\theta)-T(JX,JY,\theta)=-3d^{\sb}\theta(X,Y)=-3(d\mu\lrcorner\sp)(X,Y)=-3(d\lambda\lrcorner\ps)(X,Y)=0.
\end{equation}
We have from \eqref{li},  the symmetricity of $\sb \theta$ and \eqref{dt2} that
\begin{multline*}
(\sb_X\theta)\theta=(\sb_{\theta}\theta)X=-\frac13T(\theta,JX,J\theta)-\frac13T(\theta,e_i,e_j)T(JX,Je_i,e_j)=-\frac13T(\theta,e_i,e_j)T(JX,Je_i,e_j)\\=\frac13T(\theta,Je_i,e_j)T(JX,e_i,e_j)
=-\frac13T(\theta,e_i,Je_j)T(JX,e_i,e_j)=-(\sb_{\theta}\theta)X=-(\sb_X\theta)\theta,
\end{multline*}
which shows $\sb||\theta||^2=2(\sb_X\theta)\theta=0$ and $|\theta|^2=|J\theta|^2=const.$

Since $|J\theta|^2=const.$ we have
\begin{equation}\label{weitz}
\begin{split}
0=\nabla_{e_i}\nabla_{e_i}||J\theta||^2=(\nabla_{e_i}\nabla_{e_i}J\theta)(e_j).J\theta(e_j)+||\nabla J\theta||^2.
\end{split}
\end{equation}
We use the Killing condition of $J\theta$ and the Ricci identities for $\sb$ to evaluate the first term in \eqref{weitz}
\begin{multline}\label{w11}
(\nabla_{e_i}\nabla_{e_i}J\theta)(e_j).J\theta(e_j)=-(\nabla_{e_i}\nabla_{e_j}J\theta)(e_i).J\theta(e_j)\\
=-(\nabla_{e_j}\nabla_{e_i}J\theta)(e_i).J\theta(e_j)+R(e_i,e_j,e_i,e_k)J\theta(e_k)J\theta(e_j)+T(e_i,e_j,e_k)(\nabla_{e_k}J\theta)e_i.J\theta(e_j)\\=-Ric(J\theta,J\theta)+T(e_i,J\theta,e_k)(\nabla_{e_k}J\theta)e_i=T(e_i,J\theta,e_k)(\nabla_{e_k}J\theta)e_i=-T(e_i,J\theta,e_k)(\nabla_{e_k}\theta)Je_i\\=T(Je_i,J\theta,e_k)(\nabla_{e_k}\theta)e_i.
\end{multline}
We calculate from \eqref{dt2},  \eqref{symJtheta}  and \eqref{li}
\begin{multline}\label{wt4}
0=(\nabla_{e_i}T)(X,e_i,\theta)+T(X,e_i,e_j)(\nabla_{e_i}\theta)(e_j)= (\nabla_{e_i}T)(JX,Je_i,\theta)+T(JX,Je_i,e_j)(\nabla_{e_i}\theta)(e_j)\\=\frac13\sigma^T(JX,Je_i,\theta,e_i)+T(JX,Je_i,e_j)(\nabla_{e_i}\theta)(e_j)\\=\frac13\Big[T(JX,Je_i,e_k)T(e_k,\theta,e_i)+T(Je_i,\theta,e_k)T(e_k,JX,e_i)+T(\theta,JX,e_k)T(e_k,Je_i,e_i)\Big]\\
-T(Je_i,JX,e_j)(\nabla_{e_j}\theta)(e_i)=\frac23\Big[T(JX,Je_i,e_k)T(e_k,\theta,e_i)-T(\theta,JX,J\theta)\Big]-T(Je_i,JX,e_j)(\nabla_{e_j}\theta)(e_i)\\=\frac23T(JX,Je_i,e_k)T(e_k,\theta,e_i)-T(Je_i,JX,e_j)(\nabla_{e_j}\theta)(e_i).
\end{multline}
Set $X=\theta$ into \eqref{wt4} to get applying \eqref{dt2}
\begin{multline*}
3T(Je_i,J\theta,e_j)(\nabla_{e_j}\theta)(e_i)=2T(J\theta,Je_i,e_k)T(\theta,e_i,e_k)=-2T(J\theta,e_i,e_k)T(\theta,Je_i,e_k)\\=2T(J\theta,e_i,e_k)T(\theta,e_i,Je_k)=-2T(J\theta,e_i,Je_k)T(\theta,e_i,e_k)\\=-3T(Je_i,J\theta,e_j)(\nabla_{e_j}\theta)(e_i).
\end{multline*}
Hence,
\begin{equation}\label{wt6}
T(Je_i,J\theta,e_j)(\nabla_{e_j}\theta)(e_i)=0.
\end{equation}
Substitute \eqref{wt6} into \eqref{weitz} to get $\sb J\theta=0=\sb\theta$ which completes the proof of Lemma~\ref{sbtheta}
\end{proof}
To finish the proof of Theorem~\ref{mainsu3} we observe  from \eqref{li}, \eqref{tit}, \eqref{lmt},  \eqref{ricdt}, \eqref{ricnew} and Lemma~\ref{sbtheta} the validity of the following  identities
\begin{equation}\label{nth}
\begin{split}
2\sb_pT_{jkl}\Phi_{jkli}=12\sb_p\theta_i=0=dT_{pjkl}\Phi_{jkli};\\
2\sb_pT_{jkl}\ps_{jkl}=12\sb_p\lambda=0=-dT_{pjkl}\ps_{jkl};\\
2\sb_pT_{jkl}\sp_{jkl}=12\sb_p\mu=0=-dT_{pjkl}\sp_{jkl}.
\end{split}
\end{equation}
The identities \eqref{nth} show that the 3-forms $(X\lrcorner\sb T)\in \Lambda^3_{12}$ and $(X\lrcorner dT)\in \Lambda^3_{12}$. Hence, the four forms $\sb T=dT=0$ due to Proposition~\ref{4-for} and \eqref{dh} implies $\sigma^T=0$. Therefore, $\LC T=0$ because of  \eqref{tgt}. Now Theorem~\ref{tFBI} shows that the Riemannian first Bianchi identity \eqref{RB} holds.

For the converse, it is a well-known algebraic fact that the Riemannian first Bianchi identity \eqref{RB} on an ACYT manifold implies \eqref{r4} and the vanishing of the Ricci tensor. This completes the proof of Theorem~\ref{mainsu3}. Theorem~\ref{cmainsu3} follows from Theorem~\ref{mainsu3} and Theorem~\ref{thnnew}.

\section{ACYT with closed torsion and generalized Ricci solitons}We begin with the next 
\begin{thrm}\label{closT}
Let $(M,g,J,\Psi)$ be an   ACYT 6-manifold. 
The following conditions are equivalent:
\begin{itemize}
\item[a).] The torsion is closed, $dT=0$;
\item[b).] The exterior derivative of the Lee form belongs to $su(3)\cong\Lambda^2_8$ (i.e. it is $J$-invariant trace-free 2-orm) and the Ricci tensor of the torsion connection  is equal to $-\sb\theta$,
\begin{equation}\label{clos1}d\theta(X,Y)=d\theta(JX,JY), \quad d\theta\lrcorner F=d\theta\wedge\Phi=0,  \qquad Ric=-\sb\theta;
\end{equation}
\item[c).] The  Nijenhuis tensor is parallel with respect to the torsion connection and the Ricci tensor of the torsion connection  is equal to $-\sb\theta$,
\begin{equation}\label{clos2}\sb N=0, \qquad  Ric=-\sb\theta.
\end{equation}
\end{itemize}
\end{thrm}
\begin{proof}
The condition $dT=0$ and the equality \eqref{ricdt} imply $Ric=-\sb\theta$. We get from \eqref{ricnew} and \eqref{lmt} $d\lambda=d\mu=0$ since the torsion is closed. Hence, c)  follows from \eqref{nn1}.

Now Proposition~\ref{tnnew}
implies the first and the third identities in \eqref{clos1}.

To show that  the second equality in \eqref{clos1}  follows from c)  we observe  $-\delta T_{ij}=Ric_{ij}-Ric_{ji}=d^{\sb}\theta_{ij}$. Consequently, we have
$d^{\sb}\theta_{ij}F_{ij}=\sb_i\theta_j F_{ij}=0$ since $\delta J\theta=\sb_i\theta_jF_{ij}=\d^2 F=0$. Hence,
\begin{equation}\label{dsu1}
0=(d^{\sb}\theta\lrcorner F).vol=-(\d T\lrcorner F). vol=\d T\wedge\Phi=\frac12\delta T\wedge F\wedge F.
\end{equation}
On the other hand, we have from \eqref{torcy} using $d\lambda=d\mu=0$ that
\begin{multline*}
d*T=d\theta\wedge F-\theta\wedge dF+\lambda d\sp-\mu d\ps\\
=d\theta\wedge F-\theta\wedge dF +\lambda\Big[\theta\wedge\sp-\mu *F\Big]-\mu\Big[\theta\wedge\ps-\lambda *F\Big]\\
=d\theta\wedge F-\theta\wedge dF +\lambda\theta\wedge\sp-\mu\theta\wedge\ps,
\end{multline*}
yielding  the identity
\begin{equation*}
\begin{split}
d*T\wedge F=d\theta\wedge F\wedge F-\theta\wedge dF\wedge F=d\theta\wedge F\wedge F+\theta\wedge(*J\theta)=d\theta\wedge F\wedge F;\\
d\theta\wedge F\wedge F=d*T\wedge F=-d*T\wedge *\Phi=\Phi\wedge\d T=0,
\end{split}
\end{equation*}
where we apply \eqref{dsu1} to obtain the last equality. Hence, b) follows from c).

For the converse, assume b) holds. The third equality  in \eqref{clos1} combined with \eqref{ricdt} yields
\[dT_{jabc}\Phi_{iabc}=0.\]
The first equality in \eqref{clos1} implies $d\lambda=d\mu=0$ due to Proposition~\ref{tnnew}. Apply \eqref{ricnew} and \eqref{lmt} to get
\[dT_{jabc}\ps_{abc}=dT_{jabc}\sp_{abc}=0.\]
The last three equalities show that the 3-form $(X\lrcorner dT)\in \Lambda^3_{12}$. Hence, the four form $dT=0$ due to Proposition~\ref{4-for} which completes the equivalences between a) and b).
\end{proof}
In view of Theorem~\ref{thnnew} we get from Theorem~\ref{closT}
\begin{cor}\label{corclosT}
Let $(M,g,J,\Psi)$ be a compact  ACYT 6-manifold.
The following conditions are equivalent:
\begin{itemize}
\item[a).] The torsion is closed, $dT=0$;
\item[b).] The Ricci tensor of the torsion connection satisfies  $Ric=-\sb\theta$.

In particular $d\theta\in su(3)\cong\Lambda^2_8$.
\end{itemize}
\end{cor}
As a consequence of Theorem~\ref{closT}, we obtain
\begin{cor}\label{cclosT}
Let $(M,g,J,\Psi)$ be a  compact  CYT 6-manifold. 
The following conditions are equivalent:
\begin{itemize}
\item[a).]  The CYT space is pluriclosed, $dT=\partial\bar\partial F=0$;
\item[b).] The Ricci tensor of the torsion connection satisfies  $Ric=-\sb\theta$.
\end{itemize}
In particular  $d\theta\in su(3)\cong\Lambda^2_8$.
\end{cor}
\subsection{Proof of Theorem~\ref{closTt}}

\begin{proof} The equivalence of a) and c)  as well as  d) and f) follow from Theorem~\ref{closT}.

The second Bianchi identity for the torsion connection reads \cite[Proposition~3.5]{IS}
\begin{equation}\label{e1}
d(Scal)_j-2\sb_iRic_{ji}+\frac16d||T||^2_i+\delta T_{ab}T_{abj}+\frac16T_{abs}dT_{jabc}=0.
\end{equation}
From Theorem~\ref{closT}  we have taking into account \eqref{rics}
\begin{equation}\label{nnewt}Ric=-\sb\theta, \quad Scal =\delta \theta, \quad \delta T_{ij}=\sb_i\theta_j-\sb_j\theta_i.
\end{equation}
Substitute  \eqref{nnewt} into \eqref{e1} and use $dT=0$ to get
\begin{equation}\label{e11}
\sb_j\delta\theta+2\sb_i\sb_j\theta_i+\delta T_{ab}T_{abj}+\frac16d||T||^2_j=0.
\end{equation}
We evaluate the second term in \eqref{e11} using the Ricci identity for $\sb$ and \eqref{nnewt} as follows
\begin{equation}\label{e111}
\sb_i\sb_j\theta_i=\sb_j\sb_i\theta_i-R_{ijis}\theta_s-T_{ijs}\sb_s\theta_i=-\sb_j\delta \theta-\sb_j\theta_s.\theta_s-\frac12\delta T_{si}T_{sij}.
\end{equation}
We obtain from \eqref{e111} and\eqref{e11}
\begin{equation}\label{tt1}
-\sb_j\delta \theta-2\sb_j\theta_s.\theta_s+\frac16\sb_j||T||^2=0.
\end{equation}
Another covariant derivative of \eqref{tt1} together with \eqref{nnewt} yield
\begin{equation}\label{tt2}
\Delta\delta\theta-2\sb_j\sb_j\theta_s.\theta_s-2||Ric||^2-\frac16\Delta||T||^2=0,
\end{equation}
where $\Delta$ is the Laplace operator $\Delta=-\LC_i\LC_if=-\sb_i\sb_if$ acting on smooth function $f$ since the torsion of $\sb$ is  a 3-form.

On the other hand,  \eqref{e1} yields
\begin{equation}\label{is1}
\sb_i\sb_j\theta_i=\sb_i(\sb_i\theta_j-\delta T_{ij})=\sb_i\sb_i\theta_j-\sb_i\delta T_{ij}=\sb_i\sb_i\theta_j-\frac12\delta T_{ia}T_{iaj},
\end{equation}
where we used the   identity  
\begin{equation}\label{iii}
\sb_i\delta T_{ij}=\frac12\delta T_{ia}T_{iaj}.
\end{equation}
shown in \cite[Proposition~3.2]{IS} for any metric connection with skew torsion.

We give a proof of \eqref{iii} for completeness.  The identity $\delta^2=0$ together with \eqref{tsym} and \eqref{tgt} imply
$$0=\delta^2 T_k=\LC_i\LC_jT_{ijk}=\LC_i\sb_jT_{ijk}=\sb_i\delta T_{ik}+\frac12T_{iks}\delta T_{is}=\sb_i\delta T_{ik}-\frac12\delta T_{is}T_{isk}.$$
The equalities \eqref{is1} and \eqref{e11} yield
\begin{equation}\label{is2}
\sb_j\delta \theta+2\sb_i\sb_i\theta_j+\frac16\sb_j||T||^2=0.
\end{equation}
The identities  \eqref{is2} and \eqref{tt2} imply
\begin{equation}\label{fmax}
\Delta\Big(\delta\theta-\frac16||T||^2 \Big)+\theta_j\sb_j\Big( \delta\theta+\frac16||T||^2\Big)=||Ric||^2\ge 0.
\end{equation}
Assume b) holds. Since $M$ is compact and $d||T||^2=0$ we may apply  the strong maximum principle (see e.g. \cite{YB,GFS}) to achieve $\delta\theta=const.=0$. Conversely, the condition $\delta\theta=0$ implies $d||T||^2=0$ by the strong maximum principle applied to \eqref{fmax}. Hence b) is equivalent to f).

Further, assume f) or b).
Then \eqref{su2} implies $d||\theta||^2=0$ since $dT=\delta\theta=0=d||T||^2$ and the norm of the  Nijenhuis tensor  is constant. We calculate using the Ricci identity for $\sb$, \eqref{nnewt} and \eqref{iii}
\begin{equation*}
\begin{split}
0=\frac12\sb_i\sb_i||\theta||^2=\theta_j\sb_i\sb_i\theta_j+||\sb\theta||^2=\theta_j\sb_i\Big(\sb_j\theta_i+\delta T_{ij}\Big)+||\sb\theta||^2\\
=-\theta_j\sb_j\delta\theta-R_{ijis}\theta_s\theta_j-\theta_jT_{ijs}\sb_s\theta_i+\theta_j\sb_i\delta T_{ij}+||\sb\theta||^2\\
=-\sb_j\theta_s\theta_s\theta_j-\frac12\theta_j\delta T_{si}T_{sij}+\frac12\theta_j\delta T_{ia}T_{iaj}+||\sb\theta||^2\\
=-\frac12\theta_j\sb_j||\theta||^2+||\sb\theta||^2=||\sb\theta||^2.
\end{split}
\end{equation*}
Hence, we show that c) follows from b).

For the converse, suppose c) holds, $\sb\theta=0$. Then b) is a consequence from \eqref{tt1} and \eqref{nnewt}.


Finally, to show the equivalences of b) and e) we use \eqref{rics} and \eqref{nnewt} to  write \eqref{fmax} in the form
\begin{equation}\label{fmaxf}
\Delta\Big(Scal^g-\frac5{12}||T||^2 \Big)+\theta_j\sb_j\Big( Scal^g-\frac1{12}||T||^2\Big)=||Ric||^2\ge 0.
\end{equation}
The strong maximum principle applied to \eqref{fmaxf} implies b) is equivalent to e) in the same way as above.

The proof of Theorem~\ref{closTt} is completed.
\end{proof}

\subsection{Steady generalized  Ricci solitons}
It is shown in \cite[Proposition~4.28]{GFS} that a Riemannian manifold $(M,g,T)$ with a closed torsion $dT=0$ is a  steady generalized  Ricci soliton if there exists a vector field $X$ and a two form $B$ such that it is a solution to the equations
\begin{equation}\label{gein3}
Ric^g=\frac14T^2-\frac12\mathbb{L}_Xg, \qquad \delta T=B
\end{equation}
for $ B$ satisfying $d(B+X\lrcorner T)=0$,  where $\mathbb{L}_X$ is the Lie derivative in the direction of $X$.
In particular $\Delta_dT=\mathbb{L}_XT$.

We also  recall  the equivalent formulation \cite[Definition~4.31]{GFS} that a compact Riemannian manifold $(M,g,T)$ with a closed 3-form $T$ is a steady generalized   Ricci soliton with $k=0$ if there exists a vector field $X$ such that
\begin{equation}\label{gein2}
Ric^g=\frac14T^2-\frac12\mathbb{L}_Xg, \qquad \delta T=-X\lrcorner T, \qquad dT=0.
\end{equation}
If the vector field $X$ is a gradient of a smooth function $f$ then one has the notion of a steady generalized gradient Ricci soliton.

It is shown in \cite[Proposition~8.14]{GFS} that any compact pluriclosed (strong) CYT space is automatically a steady generalized Ricci soliton.

We extend this result to a 6-dimensional ACYT space with a closed torsion form.
\begin{prop}\label{grsol}
Let $(M,g,J,\Psi)$ be a 6-dimensional    ACYT manifold  with closed torsion form, $dT=0$.

Then it is a steady generalized  Ricci soliton with $X=\theta$ and $B=d\theta-\theta\lrcorner T$.
\end{prop}
\begin{proof}
As we identify the vector field $X$ with its corresponding 1-form via the metric, we have using \eqref{tsym}
\begin{equation}\label{acy1}
dX_{ij}=\LC_iX_j-\LC_jX_i=\sb_iX_j-\sb_jX_i+X_sT_{sij}.
\end{equation}
In view of \eqref{rics}, \eqref{tsym} and \eqref{acy1} we write the first equation in \eqref{gein3} in the form
\begin{equation}\label{acy2}
Ric_{ij}=-\frac12\delta T_{ij}-\frac12(\sb_iX_j+\sb_jX_i)=-\frac12\delta T_{ij}-\sb_iX_j+\frac12dX_{ij}-\frac12X_sT_{sij}.
\end{equation}
Set $X=\theta, \quad B=d\theta-\theta\lrcorner T$ and use \eqref{clos1} to get $\d T=B$ and \eqref{acy2} is trivially satisfied. Hence, \eqref{gein3} holds since $d(B+\theta\lrcorner T)=d^2\theta=0$.
\end{proof}
One fundamental consequence of Perelman's energy formula for Ricci flow is that compact steady solitons for Ricci flow are automatically gradient.  Adapting these energy functionals to generalized Ricci
flow it is proved in  \cite[Chapter~6]{GFS}  that steady generalized Ricci solitons on compact manifolds are automatically gradient with k = 0, i.e there exists a smooth function $f$ such that the vector field $X$ is equal to the gradient of the function $f$  and \eqref{gein2} takes the form
\begin{equation}\label{gein1}
Ric^g_{ij}=\frac14T^2_{ij}-\LC_i\LC_j f, \qquad \delta T_{ij}=-df_sT_{sij}, \qquad dT=0.
\end{equation}
The smooth function $f$ is determined with $u=\exp(-\frac12f)$ where $u$ is  the first eigenfunction of the
Schr\"odinger operator
\begin{equation}\label{schro}-4\Delta +Scal^g-\frac1{12}||T||^2=-4\Delta +Scal+\frac16||T||^2,
\end{equation}
where the Laplace operator is defined by $\Delta=-\d d$ (see \cite[Lemma~6.3, Corollary~6.10, Corollary~6.11]{GFS}).

 In terms of the torsion connection \eqref{gein1} can be written in the form (see \cite{IS})
\begin{equation}\label{gein4}
Ric_{ij}=-\sb_i\sb_j f, \qquad \d T_{ij}=-df_sT_{sij}, \qquad dT=0.
\end{equation}

Thus, a pluriclosed CYT space is a steady generalized gradient Ricci soliton   \cite[Proposition~2.6]{GFJS}, (see also  \cite[Proposition~8.14]{GFS}).

We show that similar conclusions hold for a compact closed ACYT  manifold of dimension six.
\begin{thrm}\label{inf}
 Let $(M,g,J,\Psi)$ be a 6-dimensional  compact ACYT manifold 
 with closed torsion, $dT=0$. The next two conditions are equivalent:
 \begin{itemize}
\item[a).]  $(M,g,J,\Psi)$   is a steady generalized gradient Ricci soliton. i.e. there exists a smooth function $f$ determined with $u=\exp(-\frac12f)$ where $u$ is  the first eigenfunction of the Schr\"odinger operator \eqref{schro} such that \eqref{gein1}, equivalently \eqref{gein4},  hold.
\item[b).]  For $f$ determined  by the first eigenfunction $u$ of the  Schr\"odinger operator \eqref{schro} with  $u=\exp(-\frac12f)$, the vector field $$V=\theta-df \quad  is \quad \sb-parallel, \quad \sb V=0.$$
\end{itemize}
The $\sb-$parallel vector field $V$  determines $d\theta$ and  preserves the   $SU(3)$ structure $(g,J,F,\ps,\sp)$ and the Nijenhuis tensor, i.e
\begin{equation}\label{vsu3}
d\theta=V\lrcorner T, \quad \mathbb{L}_Vg=\mathbb{L}_VF= \mathbb{L}_VJ=\mathbb{L}_V\ps=\mathbb{L}_V\sp=\mathbb{L}_VN=0.
\end{equation}
The vector field $JV$ is also $\sb-$parallel and therefore Killing.
\end{thrm}
\begin{proof}
The first statement a) follows from Proposition~\ref{grsol} and the general considerations in \cite[Chapter~6]{GFS},  \cite[Lemma~6.3, Corollary~6.10, Corollary~6.11]{GFS}).

To prove b) follows from a) 
observe that the  condition $dT=0$ and \eqref{rics} imply
$$Ric=-\sb\theta, \quad -\d T=-d^{\sb}\theta=-d\theta+\theta\lrcorner T$$ by Theorem~\ref{closT} and \eqref{rics}. The latter  combined with \eqref{gein4} yield \[\sb(\theta-df)=0,\qquad d\theta=V\lrcorner T.\]

For the converse, b)  yields $\sb\theta=\sb df$, which, combined  with Theorem~\ref{closT}  and \eqref{rics}  gives
\[Ric_{ij}=-\sb_i\theta_j=-\sb_i\sb_jf,\quad -\d T_{ij}=Ric_{ij}-Ric_{ji}=df_sT_{sij}\]
since $0=d^2f=\LC_i\LC_jf-\LC_j\LC_if=\sb_i\sb_jf-\sb_j\sb_if+df_sT_{sij}$.
Hence, \eqref{gein4} holds which proves the equivalences between a) and b).

Further,  $\sb JV=(\sb V)J =0$. We also  have $( \mathbb{L}_Vg)_{ij}=\LC_iV_j+\LC_jV_i=\sb_iV_j+\sb_jV_i=0$ since $\sb V=0$. Hence $V$ and $JV$  are Killing vector fields.

We calculate from the definitions of the Lie derivative and the torsion tensor  (see e.g. \cite{KN}) using $\sb V=0$ that
\begin{multline}\label{lieF}
( \mathbb{L}_VF)(X,Y)=VF(X,Y)-F([V,X],Y)-F(X,[V,Y])\\=VF(X,Y)-F(\sb_VX,Y)-F(X,\sb_VY)+T(V,X,e_a)F(e_a,Y)+T(V,Y,e_a)F(X,e_a)\\=(\sb_VF)(X,Y)+d\theta(X,JY)-d\theta(Y,JX)=0,
\end{multline}
where we apply the already proved identity $d\theta=V\lrcorner T$ to get the third equality and use $\sb F=0$ and that $d\theta$ is (1,1) form due to Theorem~\ref{closT} to achieve the last identity.

Using the general identity
\[g( \mathbb{L}_VJ)X,Y)=- (\mathbb{L}_Vg)(JX,Y)- (\mathbb{L}_VF)(X,Y),\]
we conclude $ \mathbb{L}_VJ=0$.

Similarly to the proof of \eqref{lieF} we obtain
\begin{equation}\label{lieps}
\begin{split}
( \mathbb{L}_V\ps)_{ijk}=d\theta_{is}\ps_{sjk}+d\theta_{js}\ps_{ski}+d\theta_{ks}\ps_{sij};\\
( \mathbb{L}_V\sp)_{ijk}=d\theta_{is}\sp_{sjk}+d\theta_{js}\sp_{ski}+d\theta_{ks}\sp_{sij}.
\end{split}
\end{equation}
We calculate from \eqref{lieps} using the identities \eqref{iden} that
\begin{equation}\label{lieps1}
\begin{split}
( \mathbb{L}_V\ps)_{ijk}\ps_{ijk}=3d\theta_{is}\ps_{sjk}\ps_{ijk}=12d\theta_{is}\delta _{si}=0=( \mathbb{L}_V\sp)_{ijk}\sp_{ijk};\\
( \mathbb{L}_V\ps)_{ijk}\sp_{ijk}=3d\theta_{is}\ps_{sjk}\sp_{ijk}=-12d\theta_{is}F_{si}=0=-( \mathbb{L}_V\sp)_{ijk}\ps_{ijk}.
\end{split}
\end{equation}
since $d\theta\in su(3)$ due to Theorem~\ref{closT} which is used to conclude the last equality in \eqref{lieps1}.

We know that $d\theta$ is a  (1,1) two form due to Theorem~\ref{closT}  and $\ps,\sp$ are (3,0)+(0,3) three forms. Apply this to \eqref{lieps} to conclude that the three forms $( \mathbb{L}_V\ps), ( \mathbb{L}_V\sp)$ are  (3,0)+(0,3) three forms and therefore belong to $span\{\ps,\sp\}$ since M is six dimensional. Therefore, \eqref{lieps1} imply  $\mathbb{L}_V\ps=0=\mathbb{L}_V\sp$. Since $d\lambda=d\mu=0$ we get from \eqref{nn1} $\mathbb{L}_VN=0$ which completes the proof  of \eqref{vsu3}.
\end{proof}
In the  case $N=0$ one has general identity $ \mathbb{L}_{JV}J=-J \mathbb{L}_VJ$ and we have
\begin{cor}\label{infc}
Let $(M,g,J,\Psi)$ be a 6-dimensional  compact pluriclosed CYT space.  The next two conditions are equivalent:
\begin{itemize}
\item[a).]  $(M,g,J,\Psi)$   is a steady generalized gradient Ricci soliton; 
\item[b).]  The vector field $V=\theta-df $  is $ \sb-$parallel, $\sb V=0.$
\end{itemize}
The $\sb-$parallel vector field $V$  determines $d\theta$ and  preserves the   $SU(3)$ structure $(g,J,F,\ps,\sp)$, i.e
\begin{equation*}
d\theta=V\lrcorner T, \quad \mathbb{L}_Vg=\mathbb{L}_VF= \mathbb{L}_VJ=\mathbb{L}_V\ps=\mathbb{L}_V\sp=0.
\end{equation*}
The vector field $JV$ is $\sb-$parallel, Killing and holomorphic.
\end{cor}
We remark that Corollary~\ref{infc} is known in all dimensions due to \cite[Proposition~2.6]{GFJS}.  Corollary~\ref{infc} on a complex 6-manifold  is a part of the 6-dimensional case of 
\cite[Proposition~2.6]{GFJS} plus additional conclusions that the vector field $V$ is $\sb$-parallel and preserves the $SU(3)$ structure. Basic examples of compact Hermitian non-K\"ahler steady generalized gradient Ricci solitons are compact Lie groups endowed with biinvariant metric and the left-invariant Samelson's complex structure. The Strominger-Bismut connection is the left-invariant flat Cartan connection with harmonic torsion 3-form of constant norm and the function $f$ has to be a constant.



\begin{thebibliography}{99}

\bibitem{AFer} I. Agricola, A. Ferreira, \emph{Einstein manifolds with skew torsion},
The Quarterly Journal of Mathematics, \textbf{65}, Issue 3,  (2014),  717-741, https://doi.org/10.1093/qmath/hat050

\bibitem{AFF}  I. Agricola, A. Ferreira, Th. Friedrich, \emph{The classification of naturally reductive homogeneous spaces in dimensions $n\le 6$}, Diff. Geom. Appl., \textbf{39} (2015), 59-92.

\bibitem{AF} I. Agricola and T. Friedrich, \emph{A note on flat metric connections with antisymmetric torsion}, Diff. Geom. Appl. \textbf{28} (2010), 480-487.

\bibitem{AFS} B. Alexandrov, T. Friedrich and N. Schoemann, Almost Hermitian 6-manifolds
revisited, J. Geom. Phys. 53 (2005) 1.

\bibitem{AI} B.Alexandrov, S.Ivanov, \emph{Vanishing theorems on Hermitian manifolds}, Diff. Geom. Appl. \textbf{14} (2001), 251-265.

\bibitem{AOUV} D. Angella, A. Otal, L. Ugarte, R. Villacampa, \emph{On Gauduchon connections with K\"ahler-like curvature},  Comm. Anal. Geom. 30  (2022), no. 5, 961-1006.

\bibitem {BB} K. Becker, M. Becker, K. Dasgupta, P.S. Green,
{\em Compactifications of Heterotic Theory on Non-Kahler Complex
Manifolds: I}, JHEP 0304 (2003) 007.

\bibitem{BBE}  K. Becker, M. Becker, K. Dasgupta, P.S. Green, E. Sharpe,
{\emph Compactifications of Heterotic Strings on Non-Kahler Complex Manifolds: II}, Nucl. Phys. {\bf B678} (2004), 19-100.



\bibitem{BM} F. Belgun, A. Moroianu, \emph{Nearly Kdhler manifolds with reduced holonomy}, Ann. Global
Anal. Geom., 19 (2001), pp. 307-319.

\bibitem{BV} Lucio Bedulli, Luigi Vezzoni, \emph{The Ricci tensor of SU(3)-manifolds},  J. Geom. Phys. 57 (2007), n. 4, 1125-1146.

\bibitem{bismut} J. M. Bismut, \emph{A local index theorem of non-K\"ahler manifolds}, Math. Ann. \textbf{284} (1989), 681-699.

\bibitem{CS} E. Cartan, J.A. Schouten, \emph{On Riemannian manifolds admitting an absolute parallelism}, Proc. Amsterdam \textbf{29} (1926), 933-946.

\bibitem{Car}  G. L. Cardoso, G. Curio, G. Dall'Agata, D. Lust, P. Manousselis,
G. Zoupanos, {\em Non-K\"ahler String Backgrounds and their Five
Torsion Classes}, Nucl.Phys. {\bf B652} (2003) 5-34.

\bibitem{Car1}  G. L. Cardoso, G. Curio, G. Dall'Agata, D. Lust, {\em
BPS Action and Superpotential for Heterotic String Compactifications with Fluxes},
JHEP 0310 (2003) 004.

\bibitem {CSal} S. Chiossi, S. Salamon, {\em The intrinsic torsion of SU(3) and
  $G_2$-structures}, Differential Geometry, Valencia 2001, World Sci.
  Publishing, 2002, pp. 115-133.

  \bibitem{Ch1} I. Chrysikos, \emph{Invariant Connections with Skew-Torsion and $\sb$-Einstein Manifolds},  Journal of Lie Theory \textbf{26} (2016) 011-048.

\bibitem{Ch2} I. Chrysikos, Ch. O'Cadiz Gustad, H. Winther, \emph{ Invariant connections and $\sb$-Einstein structures on isotropy irreducible spaces}, J. Geom. Phys. \textbf{138} (2019) 257-284


\bibitem{CMS}R. Cleyton, A. Moroianu, U. Semmelmann, \emph{Metric connections with parallel skew-symmetric torsion}, Adv. Math. \textbf{378} (2021), Paper No. 107519, 50 pp., https://doi.org/10.1016/j.aim.2020.107519

 \bibitem{CPY1} T. C. Collins, S. Picard, S.-T. Yau, \emph{The Strominger system in the square of a K\"ahler class},  arXiv:211.03784.

 \bibitem{CPYau} T. C. Collins, S. Picard, S.-T. Yau, \emph{Stability of the tangent bundle through conifold transitions}, Comm. Pure Appl. Math. \textbf{77} (2024), no. 1, 284-371.


   \bibitem{sethi}
  K.~Dasgupta, G.~Rajesh and S.~Sethi,
  \textit{M theory, orientifolds and G-flux,}
   JHEP {\bf 08}{1999}{023}; 

 \bibitem{FHP}  T. Fei, Z. Huang, S. Picard, \emph{The Anomaly
flow over Riemann surfaces}, Int. Math. Res. Not. IMRN(2021), no. 3, 2134-2165. https://doi.org/10.1093/imrn/rnz076 



\bibitem{FHP1} T. Fei, Z. Huang, S. Picard, \emph{A
Construction of Infinitely Many Solutions to the Strominger
System},  J. Differential Geom.  Volume 117, Number 1 (2021), 23-39. 

\bibitem{FIUV} M.~ Fern\'andez, S.~ Ivanov, L.~ Ugarte, R.~ Villacampa, \emph{Non-K\"ahler Heterotic String Compactifications with
non-zero fluxes and constant dilaton},  Comm. Math. Phys. {\bf 288} (2009), 677-697.

\bibitem{salamon}
  A~ Fino, M~ Parton, S~ Salamon, \emph{Families of strong KT structures in six dimensions}, Comment. Math. Helv. \textbf{79} (2004), 317-340;


\bibitem{FT1} A. Fino, A. Tomassini, \emph{Blow-ups and resolutions of strong K\"ahler with torsion metrics}, Adv.
Math. \textbf{221} (2009), 914-935.



\bibitem{FU} A.Fino, L. Ugarte,  \textit{On generalized gauduchon metrics,} Proc. Edinb. Math. Soc. 56(3) 2013, 733-753.

\bibitem{FI} Th. Friedrich, S. Ivanov,  \emph{Parallel spinors and connections with skew-symmetric
torsion in string theory},  Asian Journ.
Math. {\bf 6} (2002), 303 - 336.

\bibitem{FI1} Th. Friedrich, S.Ivanov, {\em Killing spinor equations in
  dimension 7 and geometry of integrable $G_2$ manifolds}, J. Geom. Phys
  {\bf 48} (2003), 1-11.

\bibitem{yau}
  J.~X.~Fu and S.~T.~Yau, \emph{The theory of superstring with flux on non-K\"ahler manifolds and the complex Monge-Amp\'ere equation}, J. Differential Geom. \textbf{78} (2008), no. 3, 369-428.

\bibitem{yau1}  J.~X.~Fu and S.~T.~Yau, \emph{The theory of superstring with flux on non-K\"ahler manifolds and the complex Monge-Ampere equation}, J. Diff. Geom. \textbf{78} (2009), 369--428.


\bibitem{FWW} J.~ Fu, Z.~ Wang, and D.~ Wu, \textit{Semilinear equations, the $\gamma_k$ function, and
generalized Gauduchon metrics,} J. Eur. Math. Soc., \textbf{15} (2013), 659-680.

\bibitem{GFJS} M. Garcia-Fernandez, J. Jordan, J. Streets, \emph{Non-K\"ahler Calabi-Yau geometry and
pluriclosed flow},  J. Math. Pures Appl. (9) 177 (2023), 329-367, DOI: 10.1016/j.matpur.2023.07.002.

\bibitem{GRST}M. Garcia-Fernandez, R. Rubio, C. S. Shahbazi, C. Tipler. \emph{Canonical metrics on holomorphic Courant algebroids}, Proc. London Math. Soc. \textbf{125} (2022), 700-758.

\bibitem{GFS} M. Garcia-Fernandez and J. Streets, Generalized Ricci Flow. AMS University Lecture Series 76, 2021.

\bibitem{hull}
S.~J.~Gates, Jr., C.~M.~Hull and M.~Rocek, \emph{Twisted Multiplets and New Supersymmetric Nonlinear Sigma Models},
  Nucl.\ Phys.\ B {\bf 248} (1984) 157.


\bibitem{gauduchon}
  P.~ Gauduchon, \emph{Hermitian connections and Dirac operators}, Boll. Un. Mat. It. 11-B (1997) Suppl. Fasc.,
257-288.


  \bibitem{GPap} J. Gillard, G. Papadopoulos, D. Tsimpis, {\em
Anomaly, Fluxes and (2,0) Heterotic-String Compactifications}, JHEP 0306 (2003) 035.

\bibitem{GP} E. Goldstein, S. Prokushkin, {\em Geometric Model for Complex
Non-K\"ahler Manifolds with SU(3) Structure}, Commun. Math. Phys.
{\bf 251} (2004), 65-78.

  \bibitem{gray} A. Gray, \emph{The structure of nearly K\"ahler manifolds}, Math. Ann., 223 (1976), pp. 233-248.


\bibitem{GrH} A. Gray, L. Hervella, {\em The sixteen classes of
almost Hermitian manifolds and their linear invariants}, Ann. Mat. Pura Appl. (4) {\bf 123}
 (1980), 35--58.

 \bibitem{GLMW} S. Gurrieri, J. Louis, A. Micu, D. Waldram,
{\em  Mirror Symmetry in Generalized Calabi-Yau Compactifications},
 Nucl.Phys. {\bf B654} (2003) 61-113

\bibitem{GM}  S. Gurrieri, A. Micu, {\em Type IIB Theory
on Half-flat Manifolds}, Class.Quant.Grav. {\bf 20} (2003), 2181-2192.

\bibitem {GMW} J.Gauntlett, D.Martelli, D.Waldram, {\em Superstrings with
  Intrinsic torsion}, Phys. Rev. {\bf D69} (2004) 086002.

\bibitem {GKMW} J.Gauntlett, N. Kim, D.Martelli, D.Waldram, {\em Fivebranes
  wrapped on SLAG three-cycles and related geometry}, JHEP 0111 (2001) 018.

\bibitem {GMPW} J.P. Gauntlett, D. Martelli, S. Pakis, D. Waldram,
{\em  G-Structures and Wrapped NS5-Branes}, Commun. Math. Phys.
{\bf 247} (2004), 421-445.

\bibitem{hethor}
 J.~Gutowski and G.~Papadopoulos, \emph{Heterotic Black Horizons},
  JHEP {\bf 1007} (2010) 011

\bibitem{hethor1}
  J.~Gutowski and G.~Papadopoulos, \emph{Heterotic horizons, Monge-Ampere equation and del Pezzo surfaces},
  JHEP {\bf 1010} (2010) 084

  \bibitem {GIP} J. Gutowski, S. Ivanov, G. Papadopoulos,
{\em Deformations of generalized calibrations and compact non-Kahler manifolds
with vanishing first Chern class}, Asian J. Math. {\bf 7} (2003), 39-80.

\bibitem{HLM}  A.S. Haupt, O. Lechtenfeld, E.T. Musaev, \emph{Order $\
alpha'$ heterotic domain walls with warped nearly K\" ahler
geometry}, Journal of High Energy Physics, (JHEP), Volume 2014, article
id. $\# 152$, 28 pp.


 \bibitem{hkt}
  P.~S.~Howe and G.~Papadopoulos, \emph{Twistor spaces for HKT manifolds},
  Phys.\ Lett.\ B {\bf 379} (1996) 80

  \bibitem{howe}
  P.~S.~Howe and G.~Sierra, \emph{Two-dimensional Supersymmetric Nonlinear Sigma Models With Torsion},
  Phys.\ Lett.\ B {\bf 148} (1984) 451.

 \bibitem{Hull} C.M. Hull, \emph{Compactification of the heterotic superstrings}, Phys. Lett. B 178 (1986) 357-364

\bibitem{II} P. Ivanov and S. Ivanov, \emph{SU(3) instantons and G2, Spin(7) heterotic string solitons},
Commun. Math. Phys. 259 (2005) 79-102.

\bibitem{I} S.~Ivanov, \emph{Geometry of Quaternionic K\"ahler connections with torsion}, J. Geom.  Phys. {\bf 41} (2002), 235-257.

\bibitem {I1} S. Ivanov, {\em Connection with torsion, parallel spinors and geometry of Spin(7) manifolds},
Math. Res. Lett.  {\bf 11} (2004), no. 2-3, 171--186.


\bibitem{IP1}
  S.~Ivanov and G.~Papadopoulos,
  \textit{A no go theorem for string warped compactifications,}
  Phys.\ Lett.\  {\bf B497 } (2001)  309;

\bibitem{IP2}
S.~Ivanov and G.~Papadopoulos,
  \textit{Vanishing theorems and string backgrounds,}
  Class.\ Quant.\ Grav.\  {\bf 18 } (2001)  1089;

 \bibitem{IP3}
S.~Ivanov and G.~Papadopoulos,
  \textit{Vanishing theorems on $(l|k)$-strong K\"ahler manifolds with
torsion}, \textbf{Adv. Math.}, \emph{237} (2013), 147-164.

\bibitem{IS} S. Ivanov, N. Stanchev, \emph{The Riemannian Bianchi identities of metric connections with skew torsion and generalized Ricci solitons},  arXiv:2307.03986.


\bibitem{yaujost} J.~ Jost and S.T.~Yau, \textit{A non-linear elliptic system for maps from Hermitian to Riemannian
manifolds and rigidity theorems in Hermitian geometry,} Acta Math {\bf 170} (1993) 221; Corrigendum Acta Math {\bf 177} (1994) 307.


\bibitem{KST} S. Kachru, M.B. Schulz, P.K. Tripathy, S.P. Trivedi, {\em  New
Supersymmetric String Compactifications}, JHEP 0303 (2003) 061



\bibitem{Kir} V. Kirichenko, \emph{K-spaces of maximal rank}, (russian), Mat. Zam., 22 (1977), pp. 465-476.

\bibitem{KML}M. Klaput, A. Lukas, C,l Matti, \emph{Bundles over Nearly-Kahler Homogeneous Spaces in Heterotic String Theory},
 JHEP 1109 (2011) 100 .

 \bibitem{KLMS} M. Klaput, A. Lukas, C. Matti and E.E. Svanes, \emph{Moduli Stabilising in Heterotic
Nearly K\"ahler Compactifications}, JHEP 01 (2013) 015

\bibitem{KN} S. Kobayashi, K. Nomizu, Foundations of differential geometry. Vol. I, Wiley Classics Lib.
Wiley-Intersci. Publ.
John Wiley \& Sons, Inc., New York, 1996, xii+329 pp.

\bibitem{LLR}  M. Larfors, A. Lukas, F. Ruehle, \emph{Calabi-Yau Manifolds and SU(3) Structure}, 
Journal of High Energy Physics 2019(1), DOI: 10.1007/JHEP01(2019)171

\bibitem{LY} J. Li, S.-T. Yau, The existence of supersymmetric string theory with torsion, J. Diff. Geom. 70 (2005) 143.


\bibitem{Mal} A. I. Malcev, {\em On a class of homogeneous spaces},
reprinted in Amer. Math. Soc. Translations, Series 1, {\bf 9} (1962), 276-307.

\bibitem{NZ1}   L. Ni, F. Zheng,  \emph{A classification of locally Chern homogeneous Hermitian manifolds},  arXiv:2301.00579.

\bibitem{XS} X. de la Ossa, E. E. Svanes, \emph{Holomorphic Bundles
and the Moduli Space of N=1 Heterotic Compactifications},  J. High Energy Phys. 2014, no. 10, 123, front matter+54 pp.


\bibitem{OLS}   X. de la Ossa, M. Larfors, E. E. Svanes,
\emph{Exploring SU(3) Structure Moduli Spaces with Integrable G2
Structures}, Adv. Theor. Math. Physics, Volume 19 (2015) Number 4, 837-903.

\bibitem{Pap} G. Papadopoulos {\em (2,0)-supersymmetric sigma models and
 almost complex structures}, Nucl.Phys. {\bf B448} (1995), 199-219.

  \bibitem {PPZ1}  D. H. Phong, S. Picard, X. Zhang,
\emph{Anomaly flows},  Comm. Anal. Geom. 26 (2018), no. 4, 955-1008.

\bibitem{PPZ2}  D. H. Phong, S. Picard, X. Zhang, \emph{The
anomaly flow and the Fu-Yau equation},  Ann. PDE 4 (2018), no. 2, Paper No. 13, 60 pp.

\bibitem{PPZ}  D. H. Phong, S. Picard, X. Zhang,
\emph{Geometric flows and Strominger systems}, Math. Z. 288 (2018), no. 1-2, 101-113.

\bibitem{PPZ3} D. H. Phong, S. Picard, X. Zhang, \emph{The Fu-Yau equation with negative slope parameter},
Invent. Math.  \textbf{209} 2 (2017), doi:10.1007/s00222-016-0715-z

\bibitem{PPZ4}  D. H. Phong, S. Picard, X. Zhang, \emph{New curvature flows in complex geometry},  Surveys in Differential Geometry
2017. Celebrating the 50th anniversary of the Journal of Differential Geometry, 331-364, Surv. Differ.
Geom., 22, Int. Press, Somerville, MA, 2018. 

\bibitem{Ph}D. H. Phong, \emph{Geometric flows from unified string theories},  Contribution to Surveys in Differential Geometry, "Forty Years of Ricci flow", edited by H.D. Cao, R. Hamilton, and S.T. Yau, arXiv:2304.02533.






\bibitem{Pop} D. Popovici, \emph{Limits of projective manifolds under holomorphic deformations: Hodge numbers and strongly Gauduchon
metrics}, Invent. Math. \textbf{194} (2013), no. 3, 515-534.

\bibitem{Scho} N. Schoemann, Almost Hermitian structures with parallel torsion, J. Geom. Phys.
57 (2007) 2187.

\bibitem{Str1} J. Streets, \emph{Classification of solitons for pluriclosed flow on complex surfaces},
Math. Ann. 375 (2019), no. 3-4, 1555-1595.

 \bibitem{streets} J. Streets, G. Tian, \emph{A parabolic flow of pluriclosed metrics}, Int. Math. Res. Not. IMRN \textbf{16} (2010), 3101-3133.

\bibitem{ST} J. Streets, G. Tian,  \emph{Regularity results for pluriclosed flow}, Geom. Topol. \textbf{17} (2013), no. 4, 2389-2429.

\bibitem{ST1}  J. Streets, G. Tian,, \emph{Generalized K\"ahler geometry and the pluriclosed flow},
Nuclear Phys. B 858 (2012), no. 2, 366-376.

 \bibitem{Str} A.~Strominger, \emph{Superstrings With Torsion},
Nucl.\ Phys.\ B {\bf 274}, 253 (1986).


\bibitem{WYZ} Q. Wang, B. Yang, F. Zheng, \emph{On Bismut flat manifolds},  Trans. Amer. Math. Soc. 373 (2020), no. 8, 5747-5772. 


\bibitem{Yano} K. Yano. Differential geometry on complex and almost complex spaces. International Series of Monographs
in Pure and Applied Mathematics, Vol 49, A Pergamon Press Book. 1965.

\bibitem{YB} K. Yano and S. Bochner, Curvature and Betti Numbers, Ann. of
Math. Studies 32, Princeton University Press, 1953.


\bibitem{YZZ} S. T. Yau, Q. Zhao, F. Zheng, \emph{On Strominger K\"ahler-like manifolds with degenerate torsion},  Trans. Amer. Math. Soc. 376 (2023), no. 5, 3063-3085.


\bibitem{Ye} Y. Ye, \emph{ Bismut-Einstein metrics on compact complex surfaces}, arXiv:2212.04060.



\bibitem{ZZ} Q. Zhao, F. Zheng, \emph{Strominger connection and pluriclosed metrics},  J. Reine Angew. Math. 796 (2023), 245-267. 

\bibitem{ZZ1} Q. Zhao, F. Zheng, \emph{Complex nilmanifolds and K\"ahler-like connections},  J. Geom. Phys. 146 (2019), 103512, 9 pp. 




\bibitem{ZZ2} Q. Zhao, F. Zheng, \emph{On Hermitian manifolds with Bismut-Srominger parallel torsion},  arXiv:2208.03071.

\bibitem{ZZ3} Q. Zhao, F. Zheng, \emph{Bismut K\"ahler-like manifolds of dimension 4 and 5}, arXiv:2303.09267.





\end{thebibliography}
\end{document}